\documentclass[a4paper,12pt]{article}

\usepackage[T1]{fontenc}
\usepackage[utf8]{inputenc}
\usepackage[english]{babel}
\usepackage[inline,shortlabels]{enumitem}
\usepackage{amsmath}
\usepackage{amssymb}
\usepackage{amsthm}
\usepackage[noadjust]{cite}
\usepackage[a4paper,total={6in,9in}]{geometry}
\usepackage[colorlinks,citecolor=blue]{hyperref}
\usepackage[capitalize,noabbrev,nameinlink]{cleveref}
\usepackage{xcolor}
\usepackage{longtable}
\usepackage{tikz,pgfplots,pgfplotstable}
\usetikzlibrary{patterns}
\usepackage{algorithmic}
\usepackage{marginnote}
\usepackage[labelformat=simple]{subcaption}

\setlist{font=\normalfont,topsep=1ex,parsep=0ex}

\setlist[enumerate]{label=(\alph*)}

\numberwithin{equation}{section}
\numberwithin{table}{section}    
\numberwithin{figure}{section}

\crefname{figure}{Figure}{Figures}
\crefname{table}{Table}{Tables}
\crefname{assumption}{Assumption}{Assumptions}
\Crefname{ALC@unique}{Step}{Steps}

\newlist{alglist}{enumerate}{1}
\setlist[alglist]{topsep=1ex,parsep=0ex,leftmargin=*,label=\textbf{Step~\arabic*.}}

\newcommand{\R}{\mathbb{R}}

\newcommand\norm[1]{\left\Vert#1\right\Vert}
\newcommand{\N}{\mathbb{N}}

\newcommand{\K}{\mathcal{K}}

\DeclareMathOperator{\dist}{dist}

\newcommand{\dom}{\operatorname{dom}}

\newtheoremstyle{bolddef}{}{}{\normalfont}{}{\bfseries}{.}{ }{\thmname{#1}\thmnumber{ #2}\thmnote{ (#3)}}

\newtheoremstyle{boldplain}{}{}{\itshape}{}{\bfseries}{.}{ }{\thmname{#1}\thmnumber{ #2}\thmnote{ (#3)}}

\theoremstyle{bolddef}
\newtheorem{definition}{Definition}[section]
\newtheorem{algorithm}[definition]{Algorithm}
\newtheorem{assumption}[definition]{Assumption}

\theoremstyle{boldplain}
\newtheorem{lemma}[definition]{Lemma}
\newtheorem{theorem}[definition]{Theorem}
\newtheorem{proposition}[definition]{Proposition}

\newlength\figureheight
\newlength\figurewidth

\pgfplotsset{width=7cm,compat=1.3}

\usepackage{algorithm}

\definecolor{todocolor}{rgb}{1.0,0.0,0.0}

\newcommand\email[1]{\href{mailto:#1}{\texttt{#1}}}

\newcommand{\orcid}[1]{ORCID: \href{https://orcid.org/#1}{#1}}

\newcommand{\mscLink}[1]{\href{http://www.ams.org/mathscinet/msc/msc2020.html?t=#1}{#1}}

\begin{document}

\title{
	\bfseries
Convergence analysis of nonmonotone proximal gradient methods under local Lipschitz continuity and Kurdyka--{\L}ojasiewicz property
	}

\author{Xiaoxi Jia%
	\thanks{%
		Saarland University,
		Department of Mathematics and Computer Science,
		66123 Saarbrücken,
		Germany,
		\email{xiaoxijia26@163.com},
		\orcid{0000-0002-7134-2169}
	}
       \hspace*{-4mm} \and \hspace*{-4mm}
	Kai Wang%
	\thanks{%
		Nanjing University of Science and Technology,
		School of Mathematics and Statistics,
		210094 Nanjing,
		China,
		\email{wangkaihawk@njust.edu.cn}
	}
	}

\maketitle
{
\small\textbf{\abstractname.}
The proximal gradient method is a standard approach for solving composite minimization problems in which the objective function is the sum of a continuously differentiable function and a lower semicontinuous, extended-valued function.  The traditional convergence theory for both monotone and nonmonotone variants replies heavily on the assumption of global Lipschitz continuity of the gradient of the smooth part of the objective function.  Recent work has shown that monotone proximal gradient methods converge globally only when the local (rather than global) Lipschitz continuity is assumed,  provided that the Kurdyka--{\L}ojasiewicz (KL) property holds.  However,  these results have not been extended to nonmonotone proximal gradient (NPG) methods.  In this manuscript, we consider two types of NPG methods: those combined with the average line search and the max line search, respectively.  By partitioning indices into two subsets,  one of which aims to achieve a sufficient decrease in the functional sequence,  we establish global convergence and rate-of-convergence results using the local Lipschitz continuity and the KL property, without  requiring boundedness of the iterates.  While finalizing this work, we noticed that \cite{kanzow2024convergence} presented analogous results for the NPG method with average line search, but with a different partitioning strategy.  Together, we confidently conclude that the convergence theory of the NPG method is independent on index partitioning choices.
\par\addvspace{\baselineskip}
}

{
\par\addvspace{\baselineskip}
}

{
\par\addvspace{\baselineskip}
}

\section{Introduction}\label{Sec:Intro}
\indent Let us consider the following optimization problem
\begin{equation}\label{Eq:P}\tag{Q}
	\min_x \ q(x):=f(x) + g(x) \qquad \text{s.t.} \quad x \in \mathbb X,
\end{equation}
where $\mathbb X$ is an Euclidean space, i.e., real and finite-dimensional Hilbert space, $f: \mathbb X \to \mathbb R$ is continuously differentiable, $g: \mathbb X \to \overline{\mathbb R}$ is assumed to be merely lower semicontinuous.  Note that \eqref{Eq:P} is totally nonconvex, which has a wide range of applications in practice, like machine learning, image processing,  and data science \cite{BianChen2015,BrucksteinDonohoElad2009,BolteSabachTeboulleVaisbourd2018}.  \\
\indent In this manuscript, we are interested in two types of nonmonotone proximal gradient methods for solving \eqref{Eq:P}, where the nonmonotone line search techniques are based on different nonmonotone criteria. Specifically,  we consider the approach of Zhang and Hager \cite{doi:10.1137/S1052623403428208},  which monitors the weighted \textit{average} of the objective values over the past iterates, and the approach of Grippo et al. \cite{grippo1986nonmonotone}, which uses the \textit{maximum} objective value attained among the most recent iterates.  In both cases, a candidate iterate $x^{k+1} \in \dom q$ is acceptable if the corresponding objective $q(x^{k+1})$ is less than a designated merit function, defined as
\begin{equation*}
\Phi_k:=(1-p)\Phi_{k-1}+pq(x^k) \ \text{(average case)}, \quad q(x^{l(k)}):=\max_{j=0,\ldots, \min\{m, k\}} q(x^{k-j}) \ \text{(max case)},
\end{equation*}
where $m\in \mathbb N$ and $p\in (0, 1]$ are constants, $l(k)\in \{k-m, \ldots, k \}$. Note that $p$ and $m$ aim to control the level of nonmonotonicity.  For the smaller $p$, or the larger $m$, the corresponding metric function attains larger values, hence imposing a weaker condition for the acceptance of $x^{k+1}$ compared to the monotone proximal gradient methods \cite{de2023proximal}. The values of $p=1$ and $m=0$ lead to the monotone behavior, i.e., the nonmotone proximal gradient method degenerates into the monotone one, please see \cite[Algorithm~3.1]{KanzowMehlitz2022} and \cite[Algorithm~3.1]{jia2023convergence}.

The convergence theory associated with (monotone and nonmonotone) proximal gradient methods primarily requires the derivative of the smooth part of the objective function to be globally Lipschitz continuous.  However,  the requirement is quite restrictive in practical scenarios. \cite[Examples~3.6 and~3.7]{jia2023convergence}, as two classical examples, revealed that the global Lipschitz assumption on $\nabla f$ is typically violated, whereas the local Lipschitz condition is often satisfied. \cite{KanzowMehlitz2022} readdressed the classical monotone proximal gradient method \cite[Algorithm~3.1]{KanzowMehlitz2022} and the NPG method with a max line search \cite[Algorithm~4.1]{KanzowMehlitz2022}, and provided subsequential convergence results based solely on the local (not global any more) Lipschitz assumption.  Regarding the convergence of the entire sequence generated by proximal gradient methods using the local Lipschitz assumption, \cite{jia2023convergence, bonettini2018block} revealed this mystery for monotone proximal gradient methods in the presence of the Kurdyka-{\L}ojasiewicz (KL) property of the objective function.  \cite{de2023proximal} considered an adaptive nonmonotone proximal gradient scheme based on an averaged metric function and established the subsequential convergence results with the local Lipschitz assumption of $\nabla f$, where the global worst-case rates for the iterates and the subsequential stationarity were also derived.  Note that numerical behaviours and competitiveness have already been investigated in many literature,  for example see \cite{qian2024convergence, de2023proximal,doi:10.1137/22M1469663,NIPS2015_f7664060,liu2024nonmonotone, wang2024class}. In this manuscript, we emphasize theoretical aspects of the method (without numerical experiments),  showing qualitative and even quantitative properties of NPG methods with a more general framework, facilitating various applications, such as in the context of multiplier-penalty methods \cite{JiaKanzowMehlitzWachsmuth2021}.  More precisely,  we, in this manuscript, consider the convergence theory of the entire sequence generated by the NPG methods with both average and max backtracking line search techniques under the same assumptions when the KL property is employed.  

The corresponding convergence results for the NPG method with max line search have been analyzed in \cite{doi:10.1137/22M1469663}, in which the presence of an auxiliary sequence is a necessary prerequisite for convergence. However, such requirement is not required in our technique.  One the other hand, the boundedness of subdifferential of the objective function is needed in \cite{doi:10.1137/22M1469663} for all iterates, which, to the best of our knowledge,  does not hold under the local Lipschitz continuity.  In this manuscript, note that no any prerequisite is assumed and the corresponding convergence and rate-of-convergence results obtained in \cite{doi:10.1137/22M1469663} can be also achieved by merely requiring the local Lipschitz continuity of $\nabla f$.

The primary challenge posed by the NPG method, in fact,  is that the generated functional sequence may not exhibit a sufficient decrease.  To date, the sufficient descent property appears to be necessary for the technical proof when the KL property is employed \cite{AttouchBolteRedontSoubeyran2010,AttouchBolteSvaiter2013,BolteSabachTeboulle2014,BolteSabachTeboulleVaisbourd2018}.  To the best of our knowledge,  nearby all existing studies addressing this issue reply on partitioning the index sets \cite{doi:10.1137/22M1469663, kanzow2024convergence, qian2024convergence}.  Specifically, these works divide the indices into two subsets: one subset ensures a sufficient decrease in the functional sequence, while the complement accounts for the remaining case.  A recent paper \cite{kanzow2024convergence} considered the convergence results of the NPG method with average line search under the KL property, merely requiring the local Lipschitz continuity of $\nabla f$,  where the index sets are defined as follows:
\begin{equation*}
K_1:=\{k\in \mathbb N \,|\, q(x^k) \leq \Phi_{k+\bar m}\}, \quad K_2:=\{k\in \mathbb N \,|\, q(x^k) > \Phi_{k+\bar m}\},
\end{equation*}
where $\bar m$ is a predetermined integer that relies on $p_{\min}$ to facilitate the convergence analysis.  This partitioning originates from \cite{qian2024convergence} and the subsequent convergence analysis heavily depends on the techniques presented therein, where the key component is exploiting the quantitative relationship between $\sum_{i=k}^{k+\bar m-1} \sqrt{\Phi_{i-1}-\Phi_{i}}$ and $\sqrt{\Phi_{i-1}-\Phi_{i}}$, see \cite[Lemma~4.4]{kanzow2024convergence} and \cite[Theorem~2.1]{qian2024convergence}. To be honest, we and \cite{kanzow2024convergence} happened to consider the same topic almost simultaneously.  In this manuscript, for the NPG method with average line search, to the best of our knowledge, we are the first to propose a different partitioning of the index sets from \cite{qian2024convergence,kanzow2024convergence}, as follows:
\begin{equation*}
S:=\{k\in \mathbb N \,|\, q(x^k)- \Phi_{k+1}\leq \frac{\mu}{2}\|x^{k+1}-x^k\|^2\}, \quad \overline S:=\mathbb N \setminus S,
\end{equation*}
where $\mu \in (0, \frac{1}{2}\delta {p_{\min}\gamma_{\min}}]$ is a constant, and $\delta$, $p_{\min}$, and $\gamma_{\min}$ are parameters derived from the algorithm.  Actually, different partitionings of the index sets necessitate different techniques for proofs.  However, our approach yields more natural and concise proofs than those in \cite{kanzow2024convergence},  because it eliminates the need for the mentioned-above quantitative relationship (\cite[Lemma~4.4]{kanzow2024convergence} and \cite[Theorem~2.1]{qian2024convergence}).  
Furthermore, \cite{kanzow2024convergence} presents the theoretical analysis of the NPG method with a max line search for future work, we conduct the comprehensive analysis in this manuscript. The decision to present this part of work (the NPG method with average line search), despite its similarities to \cite{kanzow2024convergence}, is also justified by a noteworthy finding: both works demonstrate that the convergence results of NPG methods are independent on the specific partitioning of the index sets. This insight significantly advances the understanding of the methods.

This manuscript is organized as follows: We first recall some background knowledge in \cref{Sec: Pre}. We then present the basic properties and convergence results of NPG methods with average line search and max line search in \cref{Sec:Alg and Res} and \cref{Sec:max}, respectively.  \cref{Sec:con} concludes this manuscript.

\section{Preliminaries}\label{Sec: Pre}
Throughout the paper, the Euclidean space $\mathbb X$ is equipped with the inner product 
$\langle\cdot,\cdot\rangle\colon \mathbb X\times\mathbb X\to\R$
and the associated norm $\norm{\cdot}$. 
Denote
\[
	\dist(x,A):=\inf\{\norm{y-x}\,|\,y\in A\}
\]
as the distance of the point $x$ to the set $A$ with
$\dist(x,\emptyset):=\infty$.
For any given $\varepsilon>0$,  denote $B_\varepsilon(x):=\{y\in\mathbb X\,|\,\norm{y-x}\leq\varepsilon\}$
as the closed $\varepsilon$-ball centered at $x$.

Let the function $f\colon\mathbb X\to\R$ be continuously differentiable, the continuous linear operator
$f'(x)\colon\mathbb X\to\R$ denotes its derivative
at $x\in\mathbb X$, and we use $\nabla f(x):=f'(x)^*1$ where $f'(x)^*\colon\R\to\mathbb X$ 
is the adjoint of $f'(x)$.
So, $\nabla f$ is a mapping from $\mathbb X$ to $\mathbb X$.

A sequence $ \{ x^k \} \subset \mathbb X $ is said to
converge \emph{Q-linearly} to $ x^*\in\mathbb X $ if there is a constant
$ c \in (0,1) $ such that the inequality
\[
   \| x^{k+1} - x^* \| \leq c \| x^k - x^* \|
\]
holds for all sufficiently large $ k \in \N $. Furthermore,
$\{x^k\}$ converges \emph{R-linearly} to $ x^* $ if and only if we
have
\[
   \limsup_{k \to \infty} \| x^k - x^* \|^{1/k} < 1.
\]
Note that the R-linear convergence holds if there 
exist constants $ a > 0 $ and $ b \in (0,1) $ such that
$ \| x^k - x^* \| \leq a b^k $ holds for all sufficiently
large $ k\in\N $, in other words, $ \| x^k - x^* \| $ is 
dominated by a Q-linearly convergent null sequence.

Let a function $q\colon\mathbb X\to\overline\R$ be merely lower semicontinuous
and let $x\in\dom q$ where $\dom q:=\{x\in\mathbb X\,|\,q(x)<\infty\}$
denotes the domain of $q$. Then
\[
	\widehat\partial q(x)
	:=
	\left\{
		\eta\in\mathbb X\,\middle|\,
		\liminf\limits_{y\to x,\,y\neq x}
		\frac{q(y)-q(x)-\langle \eta,y-x\rangle}{\norm{y-x}}\geq 0
	\right\}
\] 
is called the {\em regular} (or {\em Fr\'{e}chet}) {\em subdifferential} of $q$ at $x$.
Furthermore, 
\[
	\partial q(x)
	:=
	\left\{
		\eta\in\mathbb X\,\middle|\,
		\begin{aligned}
		&\exists\{x^k\},\{\eta^k\}\subset\mathbb X\colon\\
		&\qquad
		x^k\to x,\, q(x^k)\to q(x),\,\eta^k\to\eta,\,
		\eta^k\in\widehat{\partial}q(x^k)\,\forall k\in\N
		\end{aligned}
	\right\}
\]
is denoted as the {\em limiting} (or {\em Mordukhovich}) 
{\em subdifferential} of $q$ at $x$.
Naturally, $\widehat{\partial}q(x)\subset\partial q(x)$ always holds.
Whenever $q$ is convex,  both subdifferentials coincide, i.e.,
\[
	\widehat{\partial}q(x)
	=
	\partial q(x)
	=
	\{
		\eta\in\mathbb X\,|\,
		\forall y\in\dom q \colon\,q(y)\geq q(x)+\langle\eta,y-x\rangle
	\}.
\]
By Fermat's rule \cite[Proposition~1.30(i)]{Mordukhovich2018},
whenever $x^*\in\dom q$ is a local minimizer
of $q$, then $0\in\widehat\partial q(x^*)$ hold, consequently $ 0 \in \partial q(x^*) $ holds as well.

Since $f$ is continuously differentiabble, then for any fixed $x\in\dom\phi$, the sum rule
\begin{equation}\label{eq:sum_rule}
		\partial (f+\phi)(x)
		=
		\nabla f(x)+ \partial \phi(x)
\end{equation}
holds,  see \cite[Proposition~1.30(ii)]{Mordukhovich2018}.
Fermat's rule shows that the optimality
condition
\begin{equation*}
	0 \in \nabla f (x^*)+ \partial \phi(x^*)
\end{equation*}
holds at any local minimizer $ x^*\in\dom\phi $ of \eqref{Eq:P}. Any point $ x^* \in \dom \phi $ satisfying
this necessary optimality condition is called an 
\emph{M-stationary point} of \eqref{Eq:P}.

We next introduce the Kurdyka--{\L}ojasiewicz (KL) property which plays a 
vital role in this manuscript. The 
definition stated below is a generalization of the classical 
KL inequality for nonsmooth functions as
introduced in
\cite{AttouchBolteRedontSoubeyran2010,BolteDaniilidisLewis2007,BolteDaniilidisLewisShiota2007}
and afterwards used in the convergence analysis of many
nonsmooth optimization algorithms, see
\cite{attouch2009convergence,AttouchBolteSvaiter2013,BolteSabachTeboulle2014,BotCsetnek2016,BotCsetnekLaszlo2016,Ochs2018,OchsChenBroxPock2014}
for a couple of examples.

\begin{definition}\label{Def:KL-property}
Let $ g\colon \mathbb X \to \overline{\R} $ be 
lower semicontinuous. We say that $ g $ has the 
\emph{KL property}
at $ x^* \in \{ x \in \mathbb X \,|\, \partial g(x)\neq\emptyset \} $ if there exist a constant 
$ \eta > 0 $, a neighborhood $ U\subset\mathbb X $ of $ x^* $, and a continuous
concave function $ \chi\colon[0, \eta] \to [0,\infty) $ 
which is continuously differentiable on $(0,\eta)$ and satisfies $\chi(0)=0$ as well as
$\chi'(t)>0$ for all $t\in(0,\eta)$
such that the so-called \emph{KL inequality}
\[
   \chi ' \big( g(x) - g(x^*) \big) \dist \big( 0, \partial 
   g(x) \big) \geq 1
\]
holds for all $ x \in U \cap \big\{ x \in \mathbb X \,|\, 
g(x^*) < g(x) < g(x^*) + \eta \big\} $.
The function $ \chi $ is denoted as the \emph{desingularization function}.
\end{definition}

A popular example of the desingularization function is
given by $ \chi (t) := c t^{\theta} $ for 
$ \theta \in (0,1] $ and some constant $ c > 0 $, where the
parameter $ \theta $ is called the \emph{KL exponent},
see \cite{BolteDaniilidisLewisShiota2007,Kurdyka1998}.

\section{The NPG Method with Average  Line Search and Convergence Results}\label{Sec:Alg and Res}
We first recall the nonmonotone proximal gradient method with average line search in \Cref{Alg:NonMonotoneProxGrad}.
\begin{algorithm}\caption{A Nonmonotone Proximal Gradient Method with Average Line Search}
	\label{Alg:NonMonotoneProxGrad}
	\begin{algorithmic}[1]
		\REQUIRE $\tau > 1$, $0 < \gamma_{\min} \leq  \gamma_{\max} < \infty$, 
			$\delta \in (0,1)$, $\frac{4}{5} \leq p_{\min}\leq 1$, $x^0 \in \dom q $. Choose a positive sequence $\{p_k\}$ such that $p_k \in [p_{\min}, 1]$.
		\STATE Set $k := 0$ and $\Phi_0 \leftarrow q(x^0)$.
		\WHILE{A suitable termination criterion is violated at iteration $ k $}
		\STATE Choose $ \gamma_k^0 \in [ \gamma_{\min}, \gamma_{\max}] $.
		\STATE\label{step:subproblem_solve_MonotoneProxGrad} 
			For $ i = 0, 1, 2, \ldots $, compute a solution $ x^{k,i} $ of
      		\begin{equation}\label{Eq:NonSubki}
         		\min_{x  \in \mathbb X} \ f (x^k) + \langle\nabla f(x^k), x - x^k \rangle + \frac{\gamma_{k,i}}{2} \| x - x^k \|^2 + g (x)
      		\end{equation}
      		with $ \gamma_{k,i} := \tau^i \gamma_k^0 $, until the acceptance criterion
      		\begin{equation}\label{Eq:NonStepCrit}
         		q(x^{k,i}) \leq 
         		\Phi_k - \delta \frac{\gamma_{k,i}}{2} \| x^{k,i} - x^k \|^2 
      		\end{equation}
      		holds.
		\STATE \label{item:remark} Denote $ i_k := i $ as the terminal value, and set $ \gamma_k := 
      			\gamma_{k,i_k} $ and $ x^{k+1} := x^{k,i_k} $.
                                    \STATE \label{Eq:defPhi}
                  Set 
                  $\Phi_{k+1}:=(1-p_k)\Phi_{k}+p_kq(x^{k+1})$.
      	\STATE Set $ k \leftarrow k + 1 $.
		\ENDWHILE
		\RETURN $x^k$
	\end{algorithmic}
\end{algorithm}

In order to guarantee the convergence of \Cref{Alg:NonMonotoneProxGrad}, we introduce the following requirements.
\begin{assumption}\label{Ass:ProxGradNonMonotone}
\leavevmode
\begin{enumerate}[(a)]
   \item \label{itemnon:psi_bounded} The function $ q $ is bounded from below on $ \dom g $.
   \item \label{itemnon:phi_bounded_affine} The function  $ g$ is bounded from below by an affine function.
   \item \label{itemnon:local_Lipschitz} The function
   $ \nabla f\colon \mathbb X \to \mathbb X $ is locally Lipschitz continuous.
\end{enumerate}
\end{assumption}
Note that \cref{Ass:ProxGradNonMonotone}~\ref{itemnon:psi_bounded} and \ref{itemnon:phi_bounded_affine}\ 
 are the same as  those presented in \cite[Assumption~3.1]{KanzowMehlitz2022}.  \cref{Ass:ProxGradNonMonotone}~\ref{itemnon:psi_bounded} guarantees that the optimization problem \eqref {Eq:P} is solvable.  \cref{Ass:ProxGradNonMonotone}~\ref{itemnon:phi_bounded_affine} implies that $g$ is coercive,  which consequently indicates that subproblems \eqref{Eq:NonSubki} always have solutions,  and hence the algorithm is well-defined.

The combination of \cite[Lemma~3.1]{KanzowMehlitz2022} and \cite[Lemma~4.1]{de2023proximal} implies that \Cref{Alg:NonMonotoneProxGrad} is well-defined. 
The following results are fundamental to the convergence analysis.
\begin{proposition}\label{Prof:Nonresults}   
Let \cref{Ass:ProxGradNonMonotone} hold and $\{x^k\}_{k\in \mathbb N}$ be any sequence generated by \Cref{Alg:NonMonotoneProxGrad}, then
\begin{enumerate}[(a)]
\item \label{item:decreasingPhi} The sequence $\{\Phi_k\}_{k\in \mathbb N}$ is monotonically decreasing and 
\begin{equation}\label{Eq:decreasingPhi}
q(x^{k+1})+\delta\frac{(1-p_{k})\gamma_k}{2}\|x^{k+1}-x^k\|^2 \leq \Phi_{k+1} \leq \Phi_k-\delta\frac{p_{k}\gamma_k}{2}\|x^{k+1}-x^k\|^2.
\end{equation}
\item\label{Item:level set of q} Every iterate $x^k$ remains in the sublevel set $\mathcal L_q(x^0):=\{x\in \dom q \,|\, q(x) \leq q(x^0)\} \subset \dom g$.
\item \label{item:convergence q and Phi} Both $\{q(x^k)\}_{k \in \mathbb N}$ and $\{\Phi_k\}_{k \in \mathbb N}$ converge to some finite value $q_* \geq \inf q$.
\item \label{item:convergence of sequence} $\lim_{k \to \infty} \|x^{k+1}-x^k \| =0$.
\item \label{item: Phiq}$\Phi_k \geq q(x^k)$ holds for all $k\in \mathbb N$.
\end{enumerate}
\end{proposition}
\begin{proof}
Note that \ref{item:decreasingPhi}$-$\ref{item:convergence of sequence} have been observed in \cite[Lemma~4.2]{de2023proximal} and \cite[Lemma~4.3]{de2023proximal},  so we omit the proof.  It remains to consider \ref{item: Phiq}.

When $k=0$, one has $\Phi_0=q(x^0)$ by the initialization in \Cref{Alg:NonMonotoneProxGrad}, then for any
$k>0$, 
 from \eqref{Eq:NonStepCrit},  Steps \ref{item:remark} and \ref{Eq:defPhi}, one has
\begin{equation*}
\begin{aligned}
\Phi_k &:=(1-p_{k-1})\Phi_{k-1}+p_{k-1}q(x^{k})\\ 
&\geq (1-p_{k-1})\left(q(x^{k})+\delta \frac{\gamma_{k-1}}{2}\|x^k-x^{k-1}\|^2\right)+p_{k-1}q(x^k)\\
&=q(x^k)+\delta \frac{(1-p_{k-1})\gamma_{k-1}}{2}\|x^k-x^{k-1}\|^2 \geq q(x^k).
\end{aligned}
\end{equation*}
Hence, one has
\begin{equation*}
\Phi_k \geq q(x^k) \quad \forall k\in \mathbb N.
\end{equation*}
\end{proof}
\indent With the aid of \Cref{Prof:Nonresults}, 
we will obtain the following result.  
\begin{lemma}\label{Lem:tools}
Let \cref{Ass:ProxGradNonMonotone} hold, $\{x^k\}_{k\in \mathbb N}$ be any sequence generated by \Cref{Alg:NonMonotoneProxGrad}, and $\bar x$ be an accumulation point of $\{x^k\}_{k\in \mathbb N}$. Suppose that $\{x^k\}_{k \in \K}$ is a subsequence converging to some point $\bar x$, then $\gamma_k\|x^{k+1}-x^k \| \to_{\K} 0$ holds.
\end{lemma}
\begin{proof}
If the sequence $\{\gamma_k\}_\K$ is bounded, the statement follows directly from \cref{Prof:Nonresults}~\ref{item:convergence of sequence}. It remains to consider the case where this subsequence is unbounded. Without loss of generality, we assume that $\gamma_k \to_{\K} \infty$ and then, for the trial stepsize 
$ \hat \gamma_k := \gamma_k / \tau = \tau^{i_k-1}\gamma_k^0$, 
one also has $ \hat \gamma_k \to_\K \infty $, 
whereas the corresponding trial vector $ \hat x^k := x^{k, i_k-1} $
does not satisfy the acceptance criterion from \eqref{Eq:NonStepCrit}, together with \cref{Prof:Nonresults} \ref{item: Phiq}, one has
\begin{equation}\label{Eqnon:3-12}
	q (\hat x^k) > \Phi_k - \delta \frac{\hat \gamma_k}{2}
	\| \hat x^k - x^k \|^2 \geq  q (x^{k}) - \delta \frac{\hat \gamma_k}{2}
	\| \hat x^k - x^k \|^2\quad \forall k \in \K.
\end{equation}
On the other hand, since $ \hat x^k $ solves the corresponding
subproblem \eqref{Eq:NonSubki} with $ \hat \gamma_k $, one has
\begin{equation}\label{Eqnon:3-13}
	\langle \nabla f(x^k), \hat x^k - x^k \rangle + 
\frac{\hat \gamma_k}{2} \| \hat x^k - x^k \|^2 +
	g (\hat x^k) - g (x^k) \leq 0
\end{equation}
holds for all $k\in \K$. We immediately obtain that $ \|\hat x^k -x^k\|\to_\K 0 $ (otherwise, the left-hand side in \eqref{Eqnon:3-13} goes to infinity) and consequently $\hat x^k \to_\K \bar x$ from \cref{Ass:ProxGradNonMonotone}~\ref{itemnon:phi_bounded_affine}. By \cref{Ass:ProxGradNonMonotone}~\ref{itemnon:psi_bounded},  the mean-value theorem yields that there exists $ \xi^k $ 
on the line segment connecting $ x^k $ with $ \hat x^k $ such that 
\begin{align*}
	q (\hat x^k) - q (x^k) 
	&= 
	f(\hat x^k) + g (\hat x^k)-f(x^k) - g(x^k) 
	\\
	&= 
	\langle \nabla f(\xi^k) , \hat x^k - x^k \rangle + g (\hat x^k) - g(x^k).
\end{align*}
Recall \eqref{Eqnon:3-13}, we have
\begin{equation}\label{Eq:3-14a}
	\langle \nabla f(x^k) - \nabla f(\xi^k) ,
\hat x^k - x^k \rangle + \frac{\hat \gamma_k}{2}
	\| \hat x^k - x^k \|^2 + q (\hat x^k) - q(x^k) \leq 0
\end{equation}
for all $k\in \K$.
Exploiting \eqref{Eqnon:3-12}, one therefore obtains
\begin{align*}
	\frac{\hat \gamma_k}{2} \| \hat x^k - x^k \|^2 
	& \leq - \langle \nabla f(x^k) - \nabla f(\xi^k) ,
	\hat x^k - x^k \rangle + q (x^k) - q(\hat x^k) \\
	& < \| \nabla f(x^k) - \nabla f(\xi^k ) \| \| \hat x^k -
	x^k \| + \delta \frac{\hat \gamma_k}{2} \| \hat x^k - x^k
	\|^2 ,
\end{align*}
which can be rewritten as
\[
	( 1 - \delta ) \frac{\hat \gamma_k}{2} \| \hat x^k - x^k \| 
	< \| \nabla f(x^k) - \nabla f (\xi^k) \|.
\]
Since $ \xi^k $ is an element from the line connecting
$ x^k $ and $ \hat x^k $, it follows that $ \xi^k \to_\K \bar x$. Recall that $f$ is continuously differentiable, one has ${\hat \gamma_k}\| \hat x^k - x^k \| \to_\K 0$.\\
\indent Exploiting the fact that $x^{k+1}$ and $\hat x^k$ are solutions of the subproblem \eqref{Eq:NonSubki} with stepsize $\gamma_k$ and $\hat \gamma_k$, respectively, we obtain that
\begin{equation*}
\begin{aligned}
\langle \nabla f(x^k), x^{k+1}-x^k \rangle&+\frac{\gamma_k}{2}\|x^{k+1}-x^k\|^2+g(x^{k+1}) \\
& \leq \langle \nabla f(x^k), \hat x^k-x^k \rangle+\frac{\gamma_k}{2}\|\hat x^k-x^k\|^2+g(\hat x^k)
\end{aligned}
\end{equation*}
and
\begin{equation*}
\begin{aligned}
 \langle \nabla f(x^k), \hat x^k-x^k \rangle&+\frac{\hat \gamma_k}{2}\|\hat x^k-x^k\|^2+g(\hat x^k) \\
& \leq \langle \nabla f(x^k), x^{k+1}-x^k \rangle+\frac{\hat \gamma_k}{2}\|x^{k+1}-x^k\|^2+g(x^{k+1}).
\end{aligned}
\end{equation*}
Adding the two inequalities together, then the fact that $\gamma_k=\tau \hat \gamma_k$ implies $\|x^{k+1}-x^k \| \leq \|\hat x^k-x^k \|$, and therefore,
\begin{equation*}
\gamma_k\|x^{k+1}-x^k\|=\tau \hat \gamma_k \|x^{k+1}-x^k \| \leq \tau \hat \gamma_k\|\hat x^k-x^k \| \to_\K 0.
\end{equation*}
This completes the proof.
\end{proof}
\indent By means of the proof of \Cref{Lem:tools} and the technique outlined in \cite[Lemma~4.1]{jia2023convergence}, we immediately obtain the following desired result, i.e.,  around the accumulation point $\bar x$ of $\{x^k\}_{k\in \mathbb N}$, the associated stepsize sequence remains (uniformly) bounded.  We omit the proof because it closely resembles that of \cite[Lemma~4.1]{kanzow2024convergence}.
\begin{lemma}\label{Lem:stepsize is bounded somewhere}
Let \cref{Ass:ProxGradNonMonotone} hold,  $\{x^k\}_{k \in \mathbb N}$ be any sequence generated by \Cref{Alg:NonMonotoneProxGrad}, and $\bar x$ be its accumulation point. Then, for any $\rho>0$, there is a constant $\bar \gamma_\rho>0$ (depending on $\rho$) such that $\gamma_k \leq \bar \gamma_\rho$ for all $k\in \mathbb N$ such that $x^k \in B_\rho(\bar x)$.
\end{lemma}
The following result refers to the subsequential convergence of \Cref{Alg:NonMonotoneProxGrad}. 
\begin{theorem}\label{Thm:M-sta}
Let \cref{Ass:ProxGradNonMonotone} hold. Then each accumulation point $\bar x$ of a sequence $\{x^k\}_{k\in \mathbb N}$ generated by \Cref{Alg:NonMonotoneProxGrad} is a stationary point of \eqref{Eq:P}, and $q_*=q(\bar x)$.
\end{theorem}
\begin{proof}
Let $\{x^k\}_\K$ be a subsequence converging to $\bar x$. Recall again that $x^{k+1}$ is a solution of the subproblem \eqref{Eq:NonSubki}, hence one has
\begin{equation*}
0\in \nabla f(x^k)+\gamma_k(x^{k+1}-x^k)+\partial g(x^{k+1}) \quad \forall k\in \mathbb N.
\end{equation*}
If we have $g(x^{k+1}) \to_\K g(\bar x)$, taking $k\to_\K \infty$ to the above optimality condition, then we have $0 \in \nabla f(\bar x)+\partial g(\bar x)$ from \cref{Lem:tools} and the continuity of $\nabla f$, which means that $\bar x$ is imediately a stationary point of \eqref{Eq:P} and $q_*=q(\bar x)$ holds from \cref{Prof:Nonresults}~\ref{item:convergence q and Phi} and the continuity of $f$. So, it remains to prove $g(x^{k+1}) \to_\K g(\bar x)$.\\
\indent From the fact that $ g $ is lower semicontinuous, then one has
\begin{equation}\label{eq:Nonpsi-lower-bound}
	g (\bar x) \leq \liminf_{k \to_\K\infty} g (x^{k+1})\leq \limsup_{k \to_\K\infty} g (x^{k+1}).
\end{equation}
Since again $x^{k+1}$ is a solution of the subproblem \eqref{Eq:NonSubki} with stepsize $\gamma_k$, hence one has
\begin{equation}\label{Eq:k+1 vs stationarity}
\langle \nabla f(x^k), x^{k+1}-x^k \rangle +\frac{\gamma_k}{2}\|x^{k+1}-x^k\|^2+g(x^{k+1})\leq \langle \nabla f(x^k), \bar x-x^k \rangle +\frac{\gamma_k}{2}\|\bar x-x^k\|^2+g(\bar x)
\end{equation}
for all $k \in \mathbb N$. Recall again that $\bar x$ is an accumulation point of $\{x^k\}_{k\in \mathbb N}$, hence $\gamma_k$ is finite for sufficiently large $k\in \K$ from \cref{Lem:stepsize is bounded somewhere}. Taking $k\to_\K \infty$ to \eqref{Eq:k+1 vs stationarity}, one has $\lim\sup_{k \to_\K \infty} g(x^{k+1}) \leq g(\bar x)$ from the fact that $x^k \to_\K \bar x$, the continuity of $\nabla f$, \cref{Prof:Nonresults}~\ref{item:convergence of sequence}, and \cref{Lem:tools}. In view of \eqref{eq:Nonpsi-lower-bound}, one has $g(x^{k+1}) \to_\K g(\bar x)$, this completes the proof.
\end{proof}
\Cref{Thm:M-sta} actually obtained the subsequential convergence, we next aim to give the convergence of the entire sequence in the presence of the KL property.
Let sufficiently small $ \eta > 0 $ be the corresponding 
constant from the definition of the associated desingularization function $\chi$.
In view of \cref{Prof:Nonresults}~\ref{item:convergence of sequence}, one can find a 
sufficiently large index $ \hat{k} \in \N $ such that
\begin{equation}\label{EqNon:eta}
   \sup_{k \geq \hat{k}} \| x^{k+1} - x^k \| \leq \eta .
\end{equation}
We denote 
\begin{equation}\label{EqNon:rho}
   \rho := \eta +\frac{1}{2},
\end{equation}
as well as the compact set
\begin{equation}\label{EqNon:C_rho}
	C_\rho := B_\rho(\bar x).
\end{equation}
Finally, 
let $L_\rho>0$ be the Lipschitz constant of $ \nabla f $ on $C_\rho$ from \eqref{EqNon:C_rho}. 
In view of \Cref{Lem:stepsize is bounded somewhere}, one has
\begin{equation}\label{EqNon:gamma-rho}
   \gamma_k \leq \bar \gamma_{\rho} 
   \quad\forall x^k \in C_\rho
\end{equation}
with some suitable upper bound $ \bar \gamma_{\rho} > 0 $ 
(depending on our choice of $ \rho $ from \eqref{EqNon:rho}).  
Since $p_{\min} >\frac{4}{5}$, then we denote
\begin{equation}\label{Eq:l}
l:=\frac{1}{2}-\sqrt{\frac{1-p_{\min}}{p_{\min}}}>0.
\end{equation}
In view of \cref{Prof:Nonresults} \ref{item: Phiq}, we know that $\Phi_k \geq q(x^k)$ for all $k\in \mathbb N$, however the quantitative relationship between $\Phi_{k+1}$ and $q(x^k)$ is uncertain, so in the following, we define
\begin{equation}\label{Eq:S}
S:=\{k\in \mathbb N \,|\, q(x^k)- \Phi_{k+1}\leq \frac{\mu}{2}\|x^{k+1}-x^k\|^2\},
\end{equation}
where $\mu \in (0, \frac{1}{2}\delta {p_{\min}\gamma_{\min}}]$ is a constant. Also define $\overline S:=\mathbb N \setminus S$.
\begin{lemma}\label{LemNon:beta-small}
Let \cref{Ass:ProxGradNonMonotone} hold and $ \{ x^k \}_{k\in \mathbb N} $
be any sequence generated by \Cref{Alg:NonMonotoneProxGrad}, and $\bar x$ be an accumulation point of $\{x^k\}_{k\in \mathbb N}$.
Suppose that $ \{ x^k \}_{k \in \K} $ is a subsequence converging to $ \bar x $,  and that $ q $ has the KL property at $\bar x$ with desingularization function $ \chi $. Then there is a sufficiently large constant $ k_0 \in \K \cap \overline S$
such that the corresponding constant 
\begin{equation}\label{EqNon:alpha}
\begin{aligned}
   \alpha := 
   \| x^{k_0} - \bar x \| 
   +\frac{1}{l}\sqrt{\frac{2}{\delta p_{\min}
   	\gamma_{\min}}}\sqrt{\Phi_{k_0}-q(\bar x)}+\frac{\bar \gamma_{\rho} + 
	  	 L_{\rho}}{2l}\chi\big(\Phi_{k_0}-q(\bar x)\big).
\end{aligned}
\end{equation}
satisfies $ \alpha < \frac{1}{2}$, where $ \rho>0$, $\bar\gamma_\rho>0$, and $l>0$ are
the constants defined in \eqref{EqNon:rho}, \eqref{EqNon:gamma-rho},  and \eqref{Eq:l}, respectively,
while $L_{\rho}>0$ is the Lipschitz constant of $\nabla f$ on $C_{\rho}$ from \eqref{EqNon:C_rho} and
$ \delta>0 $ as well as $ \gamma_{\min}>0 $ are the parameters from
\Cref{Alg:NonMonotoneProxGrad}.
\end{lemma}
\begin{proof}
 From $x^k \to_\K \bar x$, then there exists a sufficiently large $k_0\in \K$ satisfying $\Phi_{k_0}-q(\bar x) \geq 0$ sufficiently small deduced by \cref{Prof:Nonresults} and \cref{Thm:M-sta}. Recall that that $q$ satisfies the KL property at $\bar x$,   the corresponding disingularizatioon function is denoted as $\chi$. For such $\chi$, one has $\chi(0)=0$ and $\chi$ is continuous and increasing on its domain, therefore $\chi \big( \Phi_{k_0}-q(\bar x) \big) \geq 0$ is sufficiently small. Hence, each summand on the right-hand side \eqref{EqNon:alpha} by taking an index $k_0 \in \K$ sufficiently large yields $ \alpha < \frac{1}{2}$.
\end{proof}
Using these notations, we obtain the following result. The proof is omitted because it is similar to  \cite[Lemma~4.3]{kanzow2024convergence}.
\begin{lemma}\label{LemNon:DistanceSubgrad}
Let \cref{Ass:ProxGradNonMonotone} hold, $ \{ x^k \}_{k \in \mathbb N} $
be any sequence generated by \Cref{Alg:NonMonotoneProxGrad}, and $\bar x$ be an accumulation point of $\{x^k\}_{k\in \mathbb N}$.
Suppose that $ \{ x^k \}_{k \in \K} $ is a subsequence converging to $ \bar x $, and that $ q $ has the KL property at $\bar x$ with desingularization function $ \chi $. Then
\begin{equation*}
	\dist\big( 0, \partial q (x^{k+1}) \big) \leq 
	\big( \bar \gamma_{\rho} + L_{\rho} \big) \| x^{k+1} - 
	x^k \|
\end{equation*}
holds for all sufficiently large 
$ k \geq \hat k$ such that $ x^k \in B_{\alpha} (\bar x) $,
where $ \alpha$ denotes the constant from \eqref{EqNon:alpha},
$ \bar \gamma_{\rho}>0$ is the constant from \eqref{EqNon:gamma-rho},
and $L_{\rho}>0$ is the Lipschitz constant of $\nabla f$ on $C_\rho$
from \eqref{EqNon:C_rho}.
\end{lemma}
By employing the above results, the following theorem demonstrates that the whole sequence $\{x^k\}_{k \in \mathbb N}$ generated by \Cref{Alg:NonMonotoneProxGrad} convergences to its accumulation point $\bar x$, provided that $q$ satisfies the KL property at $\bar x$.
\begin{theorem}\label{Thmnon:GlobConv}
Let \cref{Ass:ProxGradNonMonotone} hold, $ \{ x^k \}_{k \in \mathbb N} $
be any sequence generated by \Cref{Alg:NonMonotoneProxGrad}, and $\bar x$ be an accumulation point of $\{x^k\}_{k\in \mathbb N}$.
Suppose that $ \{ x^k \}_{k \in \K} $ is a subsequence converging to $ \bar x $, and that $ q $ has the KL property at $\bar x$. Then the entire sequence $ \{ x^k \}_{k \in \mathbb N} $ converges
to $ \bar x$.
\end{theorem}

\begin{proof}
By \cref{Prof:Nonresults}~\ref{item:decreasingPhi} and \ref{item:convergence q and Phi}, one knows that the whole sequence $\{\Phi_k\}_{k \in \mathbb N}$ is monotonically decreasing and converges to $q(\bar x)$ from \cref{Thm:M-sta}. It implies that $\Phi_k \geq q(\bar x) $ holds for all $k\in \mathbb N$. If $\Phi_k=q(\bar x)$ holds for some index $k\in \mathbb N$, which, by monotonicity, implies that $\Phi_{k+1}=q(\bar x)$.  Let us now recall Step \ref{Eq:defPhi} in \Cref{Alg:NonMonotoneProxGrad}, one has
\begin{equation}\label{Eq: Phi and q}
q(x^{k+1})=\frac{1}{p_k}(\Phi_{k+1}-\Phi_k)+ \Phi_k \quad \forall k\in \mathbb N.
\end{equation}
Hence, 
the acceptance criterion \eqref{Eq:NonStepCrit} implies that
\begin{equation*}
0\leq \| x^{k+1} - x^{k} \|  \leq \sqrt{\frac{2 \big( 
   \Phi_{k} - q(x^{k+1})\big)}{\delta 
   \gamma_{\min}}} = \sqrt{\frac{ 
   2(\Phi_{k}-\Phi_{k+1})}{p_{k}\delta 
   	\gamma_{\min}}} =0,
\end{equation*}
which says that $ x^{k+1} = x^k $ holds. Since
the subsequence $ \{ x^k \}_{\K} $ are assumed to converge to $ \bar x $, this
implies that $ x^k = \bar x $ for all $ k \in \N $ sufficiently
large.  Consequently,  the entire 
(eventually constant) sequence
$ \{ x^k \}_{k\in \mathbb N} $ converges to $ \bar x $ in this situation.

It remains to consider the case where $\Phi_k>q(\bar x)$ for all $k\in \mathbb N$. Let $ \alpha $ be the constant from 
\eqref{EqNon:alpha}, and $ k_0 \in \K $ be the corresponding
iteration index used in the definition of $ \alpha $,
see \cref{LemNon:beta-small}. One then has
$ 0 < \Phi_k - q(\bar x) \leq \Phi_{k_0}- q(\bar x) $
for all $k \geq k_0 $. 
 Without loss of generality, we may
also assume that $ k_0 \geq \hat{k}$ defined by \eqref{EqNon:eta} and that $ k_0 $ is 
sufficiently large to satisfy
\begin{equation}\label{Eqnon:k0-large}
	\Phi_{k_0} < q(\bar x) + \eta.
\end{equation}
Let $\chi\colon[0,\eta]\to[0,\infty)$ be the desingularization function
which validates the KL property of $q$ at $\bar x$.
Due to $ \chi (0) = 0 $ and $ \chi '(t) > 0 $ for all $t\in(0,\eta)$, one obtains
\begin{equation}\label{Eqnon:chi-get}
	\chi \big( \Phi_k - q(\bar x) \big) \geq 0 \quad \forall k \geq k_0.
\end{equation}
We now claim that the following two statements hold for all 
$ k \geq k_0 $:
\begin{enumerate}[(a)]
	\item \label{Itemnon:Ind-1}
	   $ x^k \in B_{\alpha} (\bar x) $, 
	\item \label{Itemnon:Ind-2}
	  $ \| x^{k_0} - \bar x \| + \sum_{i=k_0}^k \| x^{i+1} - x^i \| \leq \alpha $, which is equivalent to
\begin{equation}\label{Eqnon:Ind-2}
\begin{aligned}
\sum_{i=k_0}^{k} \|x^{i+1}-x^{i}\| &\leq\frac{1}{l}\sqrt{\frac{2}{\delta p_{\min}
   	\gamma_{\min}}}\sqrt{\Phi_{k_0}-q(\bar x)}+\frac{\bar \gamma_{\rho} + 
	  	 L_{\rho}}{2l}\chi\big(\Phi_{k_0}-q(\bar x)\big).
\end{aligned}
\end{equation}
\end{enumerate}
We verify these two statements jointly by induction. For $ k = k_0 $, statement \ref{Itemnon:Ind-1} holds simply by the 
definition of $ \alpha $ in \eqref{EqNon:alpha}. Meanwhile, \eqref{Eq:NonStepCrit} and \eqref{Eq: Phi and q} implies
\begin{equation}\label{Eq:induktiv}
\begin{aligned}
   \| x^{k+1} - x^{k} \| & \leq \sqrt{\frac{2 \big( 
   \Phi_{k} - q(x^{k+1})	\big)}{\delta 
   \gamma_{\min}}} = \sqrt{\frac{ 
   2(\Phi_{k}-\Phi_{k+1})}{p_{k}\delta 
   	\gamma_{\min}}} \leq \sqrt{\frac{2 \big( 
   	\Phi_k - \Phi_{k+1} \big)}{\delta p_{\min}
   	\gamma_{\min}}}
\end{aligned}
\end{equation}
holds for all $k\in \mathbb N$.
Recall the fact that $\{\Phi_k\}_{k\in \mathbb N}$ is decreasing and bounded below by $q(\bar x)$, then \eqref{Eq:induktiv} implies that
\begin{equation*}
\| x^{k+1} - x^{k} \| \leq \sqrt{\frac{2 \big( 
   	\Phi_{k_0} - q(\bar x) \big)}{\delta p_{\min}
   	\gamma_{\min}}} \quad \forall k \geq k_0,
\end{equation*}
hence \eqref{Eqnon:Ind-2} holds for 
$ k = k_0 $ due to $1/l>1$.  Suppose that both statements hold for some $ k \geq k_0 $.  By the triangle inequality, the induction hypothesis, 
and the definition of $ \alpha $, one obtains
\[
   \| x^{k+1} - \bar x \| \leq \sum_{i=k_0}^k \| x^{i+1} - x^i \| 
   + \| x^{k_0} - \bar x \| \leq \alpha ,
\]
i.e., statement \ref{Itemnon:Ind-1} holds for $ k+1 $ in place of $k$. The verification of 
the induction step for \ref{Itemnon:Ind-2} is more involved. 

By employing the index set $S$ defined in \eqref{Eq:S}, we consider the following two cases for sufficiently large $k \geq k_0$.\\
\textbf{Case 1: $k \in S$.} We know that
\begin{equation*}
q(x^k)-\Phi_{k+1} \leq \frac{\mu}{2}\|x^{k+1}-x^k\|^2 \quad \forall k\in S.
\end{equation*}
 Then \eqref{Eq:decreasingPhi}, \eqref{Eq:S}, and \eqref{Eq: Phi and q} imply
\begin{equation*}
\begin{aligned}
\Phi_k-\Phi_{k+1} &\geq \delta \frac{p_{\min}\gamma_{\min}}{2}\|x^{k+1}-x^k\|^2 \geq \frac{\delta p_{\min}\gamma_{\min}}{\mu} \big(q(x^k)-\Phi_{k+1}\big) \geq 2\big(q(x^k)-\Phi_{k+1}\big)\\
& =2\left(\frac{1}{p_{k-1}}\big(\Phi_k-\Phi_{k-1}\big)+\Phi_{k-1}-\Phi_{k+1} \right)\\
&=2\left(\frac{1}{p_{k-1}}-1\right)\big(\Phi_k-\Phi_{k-1}\big)+2(\Phi_k-\Phi_{k+1}) \quad \forall  S \ni k \geq k_0+1,
\end{aligned}
\end{equation*}
which yields 
\begin{equation*}
\Phi_k-\Phi_{k+1} \leq 2\left(\frac{1-p_{k-1}}{p_{k-1}}\right)\big(\Phi_{k-1}-\Phi_{k}\big) \leq 2\left(\frac{1-p_{\min}}{p_{\min}}\right)\big(\Phi_{k-1}-\Phi_{k}\big)  \quad \forall  S \ni k \geq k_0+1.
\end{equation*}
Due to $k_0\in \overline S$, summation yields that
\begin{equation}\label{Eq:s}
\begin{aligned}
\sum_{{S\ni i=k_0}}^{k+1}\sqrt{\Phi_i-\Phi_{i+1}}&=\sum_{{S\ni i=k_0+1}}^{k+1}\sqrt{\Phi_i-\Phi_{i+1}}\\
&\leq \sqrt{\frac{2(1-p_{\min})}{p_{\min}}}\sum_{{S\ni i=k_0+1}}^{k+1}\sqrt{\Phi_{i-1}-\Phi_{i}}.
\end{aligned}
\end{equation}
\textbf{Case 2: $k\in \overline S$.} For all $i$ satisfying $i\geq k_0$ and $i\in \overline S$, then one has $q(x^i) > \Phi_{i+1}$, and hence $q(x^i) \geq q(x^{i+1})$ holds. 
From \cref{Prof:Nonresults} \ref{item: Phiq} and \eqref{Eqnon:k0-large}, one has that
\begin{equation}\label{Eq:relations when k notin S}
q(\bar x) < \Phi_{i+1}<q(x^i)\leq \Phi_i \leq \Phi_{k_0} < q(\bar x)+\eta \quad 
	 \forall i\geq k_0 \ \text{and} \ i\in \overline S.
\end{equation}
Recall that $ q$ has the KL property at $\bar x$, one has
\begin{equation}\label{Eqnon:3-15}
	\chi ' \big( q (x^i) - q (\bar x) \big) \dist
	\big( 0, \partial q (x^i) \big) \geq 1 \quad 
	 \forall i\geq k_0 \ \text{and} \ i\in \overline S.
\end{equation}
Since $x^i\in B_{\alpha}(\bar x)$ holds for all $i\in \{k_0, k_0+1, \ldots, k\}$ by our induction hypothesis, we can apply \cref{LemNon:DistanceSubgrad} and obtain 
\begin{equation*}
\dist
	\big( 0, \partial q (x^{i+1}) \big) \leq (\bar \gamma_{\rho} + 
	  	 L_{\rho})\|x^{i+1}-x^{i}\| \quad \forall i \in k_0, \ldots, k,
\end{equation*}
hence, one has 
\begin{equation}\label{Eq:dist}
\dist
	\big( 0, \partial q (x^i) \big) \leq (\bar \gamma_{\rho} + 
	  	 L_{\rho})\|x^{i}-x^{i-1}\| \quad \forall i \in \{k_0+1, \ldots, k+1\} \cap \overline S.
\end{equation}
Recall that $\chi$ is increasing and concave on $[0,\eta)$,  $\chi'(t)>0$ for all $t>0$, and $\chi(0)=0$, then one has
\begin{equation}\label{Eq:chi}
\begin{aligned}
\chi  \big( \Phi_{k} - q (\bar x) \big)-\chi  \big( \Phi_{k+1} - q (\bar x) \big)&\geq \chi  \big( q(x^{k})- q (\bar x) \big)-\chi  \big( \Phi_{k+1} - q (\bar x) \big) \\
& \geq \chi' \big( q(x^{k}) - q (\bar x) \big) \big( q(x^{k}) - \Phi_{k+1}\big)\\
&\geq \frac{ q(x^{k}) - \Phi_{k+1}}{(\bar \gamma_{\rho} + 
	  	 L_{\rho})\|x^{k}-x^{k-1}\|}\\
& \geq \sqrt{\frac{2}{\delta p_{\min}\gamma_{\min}}}\frac{ q(x^{k}) - \Phi_{k+1}}{(\bar \gamma_{\rho} + 
	  	 L_{\rho})\sqrt{ \Phi_{k-1}-\Phi_{k}}}\\
&:=c\frac{ q(x^{k}) - \Phi_{k+1}}{ \sqrt{\Phi_{k-1}-\Phi_{k}}}
\end{aligned}
\end{equation}
for all $\{k_0+1, \ldots, k+1\} \cap \overline S$, where $c:=\sqrt{\frac{2}{\delta p_{\min}\gamma_{\min}}}\frac{1}{(\bar \gamma_{\rho} + 
	  	 L_{\rho})}$. Meanwhile, \eqref{Eq:NonStepCrit} means that $q(x^{k+1}) \leq \Phi_k$,  it follows that 
\begin{equation}\label{Eq:betterp}
\Phi_{k+1}:=(1-p_k)\Phi_k+p_kq(x^{k+1})\leq (1-p_{\min})\Phi_k+p_{\min}q(x^{k+1}) \quad \forall k\in \mathbb N.
\end{equation}
Denote $\Delta_{i,j}=\chi \big( \Phi_{i} - q (\bar x) \big)-\chi  \big( \Phi_{j} - q (\bar x)\big)$ for convenience,  \eqref{Eq:chi} and \eqref{Eq:betterp} yield that
\begin{equation*}
\begin{aligned}
\Phi_k-\Phi_{k+1} &\leq (1-p_{\min})\Phi_{k-1}+p_{\min}q(x^k)-\Phi_{k+1}\\
& \leq (1-p_{\min})\big(\Phi_{k-1}-\Phi_{k+1}\big)+p_{\min} \big(q(x^k)-\Phi_{k+1}\big)\\
& \leq (1-p_{\min})\big(\Phi_{k-1}-\Phi_{k+1}\big)+\frac{p_{\min}}{c}\Delta_{k,k+1}\sqrt{\Phi_{k-1}-\Phi_{k}}\\
& \leq (1-p_{\min})\left(\big(\Phi_{k-1}-\Phi_{k}\big)+\big(\Phi_{k}-\Phi_{k+1}\big)\right)+\frac{p_{\min}}{c}\Delta_{k,k+1}\sqrt{\Phi_{k-1}-\Phi_{k}},
\end{aligned}
\end{equation*}
which implies that
\begin{equation*}
p_{\min} \big( \Phi_k-\Phi_{k+1}\big)\leq (1-p_{\min})\big(\Phi_{k-1}-\Phi_{k}\big)+\frac{p_{\min}}{c}\Delta_{k,k+1}\sqrt{\Phi_{k-1}-\Phi_{k}}.
\end{equation*}
Using $\sqrt{a+b} \leq \sqrt{a}+\sqrt{b}$ and $2\sqrt{ab} \leq a+b$ for all $a,b\geq 0$, we obtain
\begin{equation*}
\begin{aligned}
\sqrt{p_{\min}}\sqrt{\Phi_k-\Phi_{k+1}} &\leq \sqrt{1-p_{\min}}\sqrt{\Phi_{k-1}-\Phi_{k}}+\sqrt{p_{\min}}\sqrt{\frac{1}{c}\Delta_{k,k+1}\sqrt{\Phi_{k-1}-\Phi_{k}}}\\
& \leq \sqrt{1-p_{\min}}\sqrt{\Phi_{k-1}-\Phi_{k}}+\frac{\sqrt{p_{\min}}}{2} \left(\frac{1}{c}\Delta_{k,k+1}+\sqrt{\Phi_{k-1}-\Phi_{k}}\right),
\end{aligned}
\end{equation*}
which yields that
\begin{equation*}
\sqrt{p_{\min}}\sqrt{\Phi_k-\Phi_{k+1}} \leq \left(\frac{\sqrt{p_{\min}}}{2}+\sqrt{1-p_{\min}}\right)\sqrt{\Phi_{k-1}-\Phi_{k}}+\frac{\sqrt{p_{\min}}}{2c}\Delta_{k,k+1}
\end{equation*}
for all $\{k_0+1, \ldots, k+1\} \cap \overline S$.  Summation yields that
\begin{equation*}
\begin{aligned}
\sum_{\overline S\ni i=k_0+1}^{k+1} \sqrt{\Phi_i-\Phi_{i+1}} & \leq \sum_{\overline S\ni i=k_0+1}^{k+1} \left(\frac{1}{2}+\sqrt{\frac{1-p_{\min}}{p_{\min}}}\right) \sqrt{\Phi_{i-1}-\Phi_{i}}+\frac{1}{2c}\Delta_{i,i+1}\\
& \leq \sum_{\overline S\ni i=k_0+1}^{k+1} \left(\frac{1}{2}+\sqrt{\frac{1-p_{\min}}{p_{\min}}}\right) \sqrt{\Phi_{i-1}-\Phi_{i}}+ \sum_{ i=k_0+1}^{k+1}\frac{1}{2c}\Delta_{i,i+1}\\
& \leq \sum_{\overline S\ni i=k_0+1}^{k+1} \left(\frac{1}{2}+\sqrt{\frac{1-p_{\min}}{p_{\min}}}\right) \sqrt{\Phi_{i-1}-\Phi_{i}}+\frac{1}{2c}\chi\big(\Phi_{k_0}-q(\bar x)\big).
\end{aligned}
\end{equation*}
Due to $k_0 \in \overline S$,  we then obtain
\begin{equation}\label{Eq:overlineS}
\begin{aligned}
\sum_{\overline S\ni i=k_0}^{k+1} \sqrt{\Phi_i-\Phi_{i+1}}& \leq \sqrt{\Phi_{k_0}-\Phi_{k_0+1}}+\sum_{\overline S\ni i=k_0+1}^{k+1} \sqrt{\Phi_i-\Phi_{i+1}}\\
&\leq \sqrt{\Phi_{k_0}-q(\bar x)}+\frac{1}{2c}\chi\big(\Phi_{k_0}-q(\bar x)\big)\\
&~~~~+\sum_{\overline S\ni i=k_0+1}^{k+1} \left(\frac{1}{2}+\sqrt{\frac{1-p_{\min}}{p_{\min}}}\right) \sqrt{\Phi_{i-1}-\Phi_{i}}.
\end{aligned}
\end{equation}
Since $p_{\min} >\frac{4}{5}$, then $\sqrt{\frac{2(1-p_{\min})}{p_{\min}}}<\frac{1}{2}+\sqrt{\frac{1-p_{\min}}{p_{\min}}}$ holds.  Adding \eqref{Eq:s} and \eqref{Eq:overlineS} yields that
\begin{equation*}
\begin{aligned}
\sum_{i=k_0}^{k+1} \sqrt{\Phi_i-\Phi_{i+1}} &\leq\sqrt{\Phi_{k_0}-q(\bar x)}+\frac{1}{2c}\chi\big(\Phi_{k_0}-q(\bar x)\big)\\
&~~~~+\left(\frac{1}{2}+\sqrt{\frac{1-p_{\min}}{p_{\min}}}\right) \sum_{ i=k_0+1}^{k+1}\sqrt{\Phi_{i-1}-\Phi_{i}}.
\end{aligned}
\end{equation*}
Recall again \eqref{Eq:l}, we obtain
\begin{equation*}
\sum_{i=k_0}^{k+1} \sqrt{\Phi_i-\Phi_{i+1}} \leq \frac{1}{l}\sqrt{\Phi_{k_0}-q(\bar x)}+\frac{1}{2cl}\chi\big(\Phi_{k_0}-q(\bar x)\big).
\end{equation*}
Then, together with \eqref{Eq:induktiv}, we have
\begin{equation*}
\begin{aligned}
   \sum_{i=k_0}^{k+1}\| x^{i+1} - x^{i} \| & \leq    \sum_{i=k_0}^{k+1} \sqrt{\frac{2 \big( 
   	\Phi_i - \Phi_{i+1} \big)}{\delta p_{\min}
   	\gamma_{\min}}}\\
&\leq \frac{1}{l}\sqrt{\frac{2}{\delta p_{\min}
   	\gamma_{\min}}}\sqrt{\Phi_{k_0}-q(\bar x)}+\frac{\bar \gamma_{\rho} + 
	  	 L_{\rho}}{2l}\chi\big(\Phi_{k_0}-q(\bar x)\big).
\end{aligned}
\end{equation*}
Hence, statement \ref{Itemnon:Ind-2} holds for $k+1$ in place of $k$, and this completes the induction.\\
\indent It follows from \ref{Itemnon:Ind-1} that 
$ x^k \in B_{\alpha} (\bar x) $ for all $ k \geq k_0 $.  Taking
$ k \to \infty $ in \eqref{Eqnon:Ind-2} shows that 
$ \{ x^k \}_{k \in \mathbb N} $ is a Cauchy sequence and hence convergent.
Since $ \bar x $ is an accumulation point,
then the entire sequence $ \{ x^k \}_{k \in \mathbb N} $ converges
to $ \bar x$.
\end{proof}
\indent Note that if all indices $k \geq k_0$ belong to $S$ or if $p_k:=1$, then \Cref{Alg:NonMonotoneProxGrad} degenerates into a monotone proximal gradient method, see \cite[Algorithm~3.1]{KanzowMehlitz2022} or \cite[Algorithm~3.1]{jia2023convergence}, and the corresponding rate-of-convergence result was obtained in \cite{jia2023augmented}. So, we in the following state the rate-of-convergence result of the NPG method (in the general case).  Readers may find more details for the proof in \cite[Theorem~4.6]{kanzow2024convergence}.

\begin{theorem}\label{Thm:Rate-of-Conv}
Let \cref{Ass:ProxGradNonMonotone} hold and $ \{ x^k \}_{k\in \mathbb N} $
be any sequence generated by \Cref{Alg:NonMonotoneProxGrad}.
Suppose that $ \{ x^k \}_{k \in \K} $ is a subsequence converging to
some limit point $ \bar x $, and that $ q $ has the KL property
at $ \bar x $. Then the entire sequence $ \{ x^k \}_{k \in \mathbb N}$ converges
to $ \bar x $, and if the corresponding desingularization function
has the form $ \chi (t) = \kappa t^{\theta} $ for some $ kappa > 0 $ and $\theta \in (0,1]$, then
the following statements hold:
\begin{enumerate}[(i)]
          \item  if $\theta=1$, then the sequences $\{\Phi_k\}_{k\in \mathbb N}
$ and $\{x^k\}_{k\in \mathbb N}$ converge in a finite number of steps to $q(\bar x)$ and $\bar x$, respectively.
         \item  if $\theta\in [\frac{1}{2}, 1)$, then $\{\Phi_k\}$ Q-linearly convergent to $q(\bar x)$, and $\{x^k\}$ R-linearly convergent to $\bar x$.
          \item  if $\theta\in (0,\frac{1}{2})$, then there exist some
           positive constants $\eta_1$ and $\eta_2$ such that
           \begin{align*}
           \Phi_k-q(\bar x) &\leq \eta_1 k^{-\frac{1}{1-2\theta}},\\
          \|x^k-\bar x\| &\leq \eta_2 { k}^{-\frac{\theta}{1-2\theta}}
             \end{align*}
           for sufficiently large $k\geq k_0$.
\end{enumerate}
\end{theorem}
\section{The NPG method with Max Line Search and Convergence Analysis}\label{Sec:max}
This section focuses on another type of the nonmonotone proximal gradient method, where the line search is chosen as the max-type rule.  Let us now introduce the algorithm which is derived from \cite[Algorithm~4.1]{KanzowMehlitz2022}.
\begin{algorithm}\caption{A Nonmonotone Proximal Gradient Method with Max Line Search }
	\label{Alg:maxNonMonotoneProxGrad}
	\begin{algorithmic}[1]
		\REQUIRE $\tau > 1$, $0 < \gamma_{\min} \leq  \gamma_{\max} < \infty$, $m \in \mathbb N$,
			$\delta \in (0,1)$, $x^0 \in \dom q$.
		\STATE Set $k := 0$.
		\WHILE{A suitable termination criterion is violated at iteration $ k $}
		\STATE Set $m_k:=\min\{k, m\}$ and choose $ \gamma_k^0 \in [ \gamma_{\min}, \gamma_{\max}] $.
		\STATE\label{step:subproblem_solve_maxMonotoneProxGrad} 
			For $ i = 0, 1, 2, \ldots $, compute a solution $ x^{k,i} $ of
      		\begin{equation}\label{Eq:maxNonSubki}
         		\min_x \ f (x^k) + \langle\nabla f(x^k), x - x^k \rangle + \frac{\gamma_{k,i}}{2} \| x - x^k \|^2 + g (x), 
         		\quad x \in \mathbb X
      		\end{equation}
      		with $ \gamma_{k,i} := \tau^i \gamma_k^0 $, until the acceptance criterion
      		\begin{equation}\label{Eq:maxNonStepCrit}
         		q(x^{k,i}) \leq 
         		\max_{j=0,1,\ldots, m_k}  q(x^{k-j})- \delta \frac{\gamma_{k,i}}{2} \| x^{k,i} - x^k \|^2 
      		\end{equation}
      		holds.
		\STATE \label{item':remark} Denote by $ i_k := i $ the terminal value, and set $ \gamma_k := 
      			\gamma_{k,i_k} $ and $ x^{k+1} := x^{k,i_k} $.
      	\STATE Set $ k \leftarrow k + 1 $.
		\ENDWHILE
		\RETURN $x^k$
	\end{algorithmic}
\end{algorithm}

Throughout this section, let us define $l(k)\in \{k-m_k, \ldots, k \}$ as an index satisfying
\begin{equation*}
q(x^{l(k)})= \max_{j=0,1, \ldots, m_k} q(x^{k-j})
\end{equation*}
for each $k\in \mathbb N$.  In order to ensure the convergence of  \cref{Alg:maxNonMonotoneProxGrad},  we need both \Cref{Ass:ProxGradNonMonotone} and 
\begin{assumption} \label{Ass:Cont}
The function $q$ is continuous on $\dom q$.
\end{assumption}
Note that \Cref{Ass:ProxGradNonMonotone}\ \ref{itemnon:phi_bounded_affine} and \Cref{Ass:Cont} imply that $q$ is uniformly continuous on the corresponding sublevel $\mathcal L_q(x^0)$. The requirement of uniform continuity about the objective function plays an essential role in the context of nonmonotone line search rules with the maximum taste \cite{grippo1986nonmonotone}.  We here recall some results for the convenience of readers. 
\begin{proposition}\label{Pro:maxprevious}
Let \cref{Ass:ProxGradNonMonotone} and \cref{Ass:Cont} hold and let $\{x^k\}_{k\in \mathbb N}$ be any sequence generated by \Cref{Alg:maxNonMonotoneProxGrad},  then
\begin{enumerate}[(a)]
\item \label{item:maxdecreasingPhi} the sequence $\{q(x^{l(k)})\}_{k\in \mathbb N}$ is monotonically decreasing,
\item\label{Item:maxlevel set of q} the sequences $\{x^k\}_{k\in \mathbb N}$, $\{x^{l(k)}\}_{k\in \mathbb N}$ $\subset \mathcal L_q(x^0) \subset \dom g$,
\item \label{item:maxconvergence of sequence} $\lim_{k \to \infty} \|x^{k+1}-x^k \| =0$,
\item \label{item:maxconvergence of gamma yimes sequence} suppose $\bar x$ is an accumulation point of $\{x^k\}_{k\in \mathbb N}$ such that $x^k \to_\K \bar x$ holds. Then $\bar x$ is an M-stationary point of \eqref{Eq:P}, and $\gamma_k\|x^{k+1}-x^k\| \to_\K 0$ is valid.
\end{enumerate}
\end{proposition}
\indent With the aid of \cref{Pro:maxprevious}~\ref{item:maxconvergence of gamma yimes sequence} as well as the techniques in \cref{Lem:stepsize is bounded somewhere}, we can obtain the following result.
\begin{lemma}\label{Lem:maxboundedgamma}
Let \cref{Ass:ProxGradNonMonotone} and \cref{Ass:Cont} hold, let $\{x^k\}_{k\in \mathbb N}$ be any sequence generated by \Cref{Alg:maxNonMonotoneProxGrad},  $\bar x$ is an accumulation point of $\{x^k\}_{k\in \mathbb N}$ and consequently let $\{x^k\}_{k\in \K}$ be a subsequence converging to some point $\bar x$. Then, for any $\rho>0$, there is a constant $\bar \gamma_\rho>0$ (depending on $\rho$) such that $\gamma_k \leq \bar \gamma_\rho$ for all $k\in \mathbb N$ such that $x^k \in B_\rho(\bar x)$.
\end{lemma}
We now discuss the convergence properties of the sequence generated by \Cref{Alg:maxNonMonotoneProxGrad} in the presence of the KL property of $q$ at the mentioned accumulation point $\bar x$. 
Let sufficiently small 
$\eta>0$ be the corresponding constant from the definition of the associated desingularization function $\chi$.  In view of \cref{Lem:maxboundedgamma}, we can find a sufficiently large index $\hat k \in \mathbb N$ such that
\begin{equation}\label{Eq:maxeta}
\sup_{k \geq \hat k-m-1} \|x^{k+1}-x^k\| \leq \eta.
\end{equation}
Let us now fix a constant $\mu \in (0, \delta \gamma_{\min})$, and define
\begin{equation}\label{Eq:definition for a}
K:=\big\{k \in \mathbb N \,|\,  q(x^{l(k+1)})-q(x^{k+1}) > \frac{\mu}{2}\|x^{k+1}-x^k\|^2\big\}
\end{equation}
and $\overline K:=\mathbb N \setminus K$. 
Define 
\begin{equation}\label{Eq:maxrho}
\rho:=(m+1)\eta+\frac{1}{2},
\end{equation}
as well as the compact set 
\begin{equation}\label{Eq:maxC_rho}
C_{\rho}:=B_{2\rho}(\bar x) \cap \mathcal L_q(x^0).
\end{equation}
Consequently, let $L_{\rho}>0$ be the (global) Lipschitz constant of $\nabla f$ on $C_{\rho}$ from \eqref{Eq:maxC_rho}. In view of \cref{Lem:maxboundedgamma}, let $\bar \gamma_{\rho}>0$ satisfy
\begin{equation}\label{Eq:maxgamma_rho}
\gamma_k \leq \bar \gamma_{\rho} \quad \forall x^k\in C_{\rho}.
\end{equation} 
\begin{proposition}\label{Prop:Cons}
Let \cref{Ass:ProxGradNonMonotone} and \cref{Ass:Cont} hold, and let $ \{ x^k \}_{k\in \mathbb N} $
be any sequence generated by \Cref{Alg:maxNonMonotoneProxGrad}. Then there exists a constant $c_{\hat k}>0$ such that
\begin{equation*}
\|x^{l(k+1)}-x^{k+1}\| \leq c_{\hat k} \|x^{l(k+1)}-x^{l(k+1)-1}\| \quad \forall k \geq \hat k,
\end{equation*}
where $\hat k$ is defined in \eqref{Eq:maxeta}.
\end{proposition}
\begin{proof}
By contradiction, for arbitrary $c>0$, there exists a $\bar k \geq \hat k$ such that
\begin{equation*}
\|x^{l(\bar k+1)}-x^{\bar k+1}\| > c \|x^{l(\bar k+1)}-x^{l(\bar k+1)-1}\|.
\end{equation*}
Meanwhile,  for the sufficiently large $\bar k \geq \hat k$, from \eqref{Eq:maxeta}, one has
\begin{equation*}
\|x^{l(\bar k+1)}-x^{\bar k+1}\| \leq \sum_{i=\bar k+1-m}^{\bar k+1} \|x^{i+1}-x^i\| \leq m\eta.
\end{equation*}
Therefore, one has
\begin{equation*}
m \eta > c \|x^{l(\bar k+1)}-x^{l(\bar k+1)-1}\|.
\end{equation*}
Now, choose \( c := \frac{1}{\|x^{l(\bar{k}+1)} - x^{l(\bar{k}+1)-1}\|} \), we obtain:
\[
m \eta > 1,
\]
which contradicts the fact that \( \eta \) is sufficiently small and \( m > 0 \) is a fixed number. 
\end{proof}
Thanks to the above notation, 
we can identify a new neighborhood centered at the accumulation point, with a very small radius, which is primarily used to to ensure the convergence of the whole sequence under the assumption of the KL property of $q$, we now introduce such radius.  
\begin{lemma}\label{LemNon:maxalpha-small}
Let \cref{Ass:ProxGradNonMonotone} and \cref{Ass:Cont} hold, and let $ \{ x^k \}_{k\in \mathbb N} $
be any sequence generated by \Cref{Alg:maxNonMonotoneProxGrad}.
Suppose that $ \{ x^k \}_{k\in \K} $ is a subsequence converging to
some limit point $ \bar x $, and that $ q $ has the KL property
at $\bar x$ with desingularization function $ \chi $. Then there is a sufficiently large constant $ k_0 \in \K \cap \overline K $
such that $k_0-m-1$ is also sufficiently large,  where $K$ is defined by \eqref{Eq:definition for a}, then the corresponding constant is defined as
\begin{equation}\label{EqNon:maxalpha}
\begin{aligned}
   \alpha &:= 
   \| x^{k_0} - \bar x \| \\
& \leq \left( \frac{8\sqrt{2}m\left(c_{\hat k}+1\right)(\bar \gamma_{\rho}+L_{\rho})}{3\delta \gamma_{\min}}\sqrt{\frac{\bar \gamma_{\rho}+L_{\rho}}{\mu}}+\frac{m}{3}\sqrt{\frac{2(\bar \gamma_{\rho}+L_{\rho})}{\delta \gamma_{\min}-\mu}}\left(\frac{3}{m}+\frac{2(\bar \gamma_{\rho}+L_{\rho})}{\delta \gamma_{\min}}\right)\right)\\
&~~~~\cdot \chi\big(q(x^{l(k_0)})-q(\bar x)\big)\\
&~~~~+ \left(\frac{4\sqrt{2}m\left(c_{\hat k}+1\right)}{3}\sqrt{\frac{\bar \gamma_{\rho}+L_{\rho}}{\mu}}+\frac{m}{3}\sqrt{\frac{2(\bar \gamma_{\rho}+L_{\rho})}{\delta \gamma_{\min}-\mu}}+\sqrt{\frac{\delta \gamma_{\min}}{\delta \gamma_{\min}-\mu}}\right)\\
&~~~~\cdot\sqrt{\frac{2\big(q(x^{l(k_0-m-1)})-q(\bar x)\big)}{\delta \gamma_{\min}}}\\
\end{aligned}
\end{equation}
satisfies $ \alpha < \frac{1}{2}$, where $ \rho>0$, $\bar\gamma_\rho>0$, $\mu>0$ are
the constants defined in \eqref{Eq:maxrho}, \eqref{Eq:maxgamma_rho},  and \eqref{Eq:definition for a}, respectively,
while 
$ \delta>0 $ and $ \gamma_{\min}>0 $ are the parameters from
\Cref{Alg:maxNonMonotoneProxGrad}.
\end{lemma}
\begin{proof}
From the fact that $x^k \to_{\K} \bar x$, the decrease of $\{q(x^{l(k)})\}_{k\in \mathbb N}$, we have $\|x^{k_0}-\bar x\|$ is sufficiently small by the sufficiently large $k_0 \in \K$,  and $q(x^{l(k)})-q(\bar x) \to 0$ from \cref{Ass:Cont}. Then we have $\chi(q(x^{l(k)})-q(\bar x)) \to 0$ from the continuity of $\chi$ and $\chi(0)=0$. Meanwhile, because $k_0-m-1$ is sufficiently large, we have $q(x^{l(k_0-m-1)})-q(\bar x)$ is sufficiently small.  Totally, we have $\alpha<\frac{1}{2}$.
\end{proof}
Taking the similar analysis to \cref{LemNon:DistanceSubgrad}, we obtain the following result.
\begin{lemma}\label{LemNon:maxDistanceSubgrad}
Let \cref{Ass:ProxGradNonMonotone} and \cref{Ass:Cont} hold, and let $ \{ x^k \}_{k\in \mathbb N} $
be any sequence generated by \Cref{Alg:maxNonMonotoneProxGrad}.
Suppose that $ \{ x^k \}_{k \in \K} $ is a subsequence converging to
some limit point $ \bar x $, 
then
\begin{equation*}
	\dist\big( 0, \partial q (x^{k+1}) \big) \leq 
	\big( \bar \gamma_{\rho} + L_{\rho} \big) \| x^{k+1} - 
	x^k \|
\end{equation*}
holds for all sufficiently large 
$ k\geq k_0$ such that $ x^k \in B_{\rho} (\bar x) $,
where $ \rho $ denotes the constant from  \eqref{Eq:maxrho},
$ \bar \gamma_{\rho}>0$ is the constant from \eqref{Eq:maxgamma_rho},
and $L_{\rho}>0$ is the Lipschitz constant of $\nabla f$ on $C_\rho$
from \eqref{Eq:maxC_rho}.
\end{lemma}
By employing the KL property,  we next illustrate that the subsequential trajectory has a finite length.
\begin{theorem}\label{Thmnon:maxauxConv}
Let \cref{Ass:ProxGradNonMonotone} and \cref{Ass:Cont} hold, and let $ \{ x^k \}_{k\in \mathbb N} $
be any sequence generated by \Cref{Alg:maxNonMonotoneProxGrad}.
Suppose that $ \{ x^k \}_{k\in \K} $ is a subsequence converging to
some limit point $ \bar x $, and that $ q $ has the KL property
at $\bar x$.  If $q(x^{l(k)}) \neq q(\bar x)$ for all $k \in \mathbb N$  and for any $v \geq k_0$ where $k_0$ is used to define $\alpha$ in \cref{LemNon:maxalpha-small},  if $x^k\in B_{\alpha}(\bar x)$ holds for all $k_0 \leq k \leq v$, then we have
\begin{equation*}
\sum_{k=k_0}^{v} \|x^{l(k)}-x^{l(k)-1}\| \leq \frac{4m}{3}\left( \sqrt{\frac{2\big(q(x^{l(k_0-m-1)})-q(\bar x)\big)}{\delta \gamma_{\min}}}+\frac{2(\bar \gamma_{\rho}+L_{\rho})}{\delta \gamma_{\min}}\chi\big(q(x^{l(k_0)})-q(\bar x)\big) \right).
\end{equation*}
\end{theorem}
\begin{proof}

Without loss of generality,  we also assume that $k_0 > \hat k+m$ (defined by \eqref{Eq:maxeta}) and $k_0$ is sufficiently large to satisfy
\begin{equation*}
q(x^{l(k_0)})<q(\bar x)+\eta,
\end{equation*}
and then 
\begin{equation}\label{Eq:maxk0}
q(\bar x)<q(x^{l(k)})\leq q(x^{l(k_0)})<q(\bar x)+\eta \quad \forall k \geq k_0.
\end{equation}
Let $\chi: [0,\eta] \to [0, \infty)$ be the desingularization function which validates of the KL property of $q$. Due to $\chi(0)=0$ and $\chi'(t)>0$ for all $t\in (0,\eta)$, one obtains
\begin{equation*}
\chi\big(q(x^{l(k)})-q(\bar x)\big) \geq 0 \quad \forall k \geq k_0.
\end{equation*}
For any $v \geq k_0$,  if $x^k \in B_{\alpha}(\bar x)$ holds for all $ k_0 \leq k \leq v$, one has  $x^{l(k)} \in B_{\alpha+m\eta}(\bar x)$ for all $k \geq k_0$ from \eqref{Eq:maxeta},
then
KL property of $q$ at $\bar x$ yields that
\begin{equation*}
\chi'\big(q(x^{l(k)})-q(\bar x)\big)\dist\big(0, \partial q(x^{l(k)})\big) \geq 1 \quad \forall k_0 \leq k \leq v,
\end{equation*}
which deduces from the concavity of $\chi$ that, 
\begin{equation}\label{Eq:KLineq}
\chi\big(q(x^{l(k)})-q(\bar x)\big)-\chi\big(q(x^{l(k+m+1)})-q(\bar x)\big) \geq \frac{q(x^{l(k)})-q(x^{l(k+m+1)})}{\dist(0, \partial q(x^{l(k)}))} \quad \forall  k_0 \leq k \leq v.
\end{equation}
By \eqref{Eq:maxNonStepCrit}, we have
\begin{equation}\label{Eq:preplusKL}
\begin{aligned}
\delta\frac{\gamma_{\min}}{2}\|x^{l(k+m+1)}-x^{l(k+m+1)-1}\|^2 &\leq q(x^{l(l(k+m+1)-1)})-q(x^{l(k+m+1)}) \\
&\leq q(x^{l(k)})-q(x^{l(k+m+1)}) \quad \forall k\in \mathbb N.
\end{aligned}
\end{equation}
Due to $\|x^{l(k)-1}-\bar x\| \leq \|x^{l(k)-1}-x^{l(k)}\|+\|x^{l(k)}-\bar x\|\leq \alpha+(m+1)\eta<\rho$ for all $k_0 \leq k \leq v$ by \eqref{Eq:maxeta},  then \cref{LemNon:maxDistanceSubgrad} implies
\begin{equation*}
\dist \big(0, \partial q(x^{l(k)}) \big) \leq (\bar \gamma_{\rho}+L_{\rho})\|x^{l(k)}-x^{l(k)-1}\|.
\end{equation*}
Then \eqref{Eq:KLineq} and \eqref{Eq:preplusKL}  imply that
\begin{equation}
\chi\big(q(x^{l(k)}-q(\bar x)\big)-\chi\big(q(x^{l(k+m+1)})-q(\bar x)\big) \geq \frac{\delta\frac{\gamma_{\min}}{2}\|x^{l(k+m+1)}-x^{l(k+m+1)-1}\|^2}{(\bar \gamma_{\rho}+L_{\rho})\|x^{l(k)}-x^{l(k)-1}\|}
\end{equation}
holds for all  $k_0 \leq k \leq v$ satisfying $x^k \in B_{\alpha}(\bar x)$.
Employing the inequality $\sqrt{ab} \leq \frac{1}{4} a+b$ for any $a \geq 0$ and $b \geq 0$ and setting $\tau:=\frac{2(\bar \gamma_{\rho}+L_{\rho})}{\delta \gamma_{\min}}$, we have
\begin{equation*}
\begin{aligned}
&\|x^{l(k+m+1)}-x^{l(k+m+1)-1}\| \\
&\leq \sqrt{\tau\|x^{l(k)}-x^{l(k)-1}\|\big(\chi\big(q(x^{l(k)}-q(\bar x)\big)-\chi\big(q(x^{l(k+m+1)})-q(\bar x)\big) \big)}\\
& \leq \frac{1}{4}\|x^{l(k)}-x^{l(k)-1}\|+\tau\big(\chi\big(q(x^{l(k)}-q(\bar x)\big)-\chi\big(q(x^{l(k+m+1)})-q(\bar x)\big) \big)
\end{aligned}
\end{equation*}
for all $k_0 \leq k \leq v$ satisfying $x^k \in B_{\alpha}(\bar x)$.
Then,  summation yields
\begin{equation*}
\begin{aligned}
&\sum_{k=k_0+m+1}^{v+m+1} \|x^{l(k)}-x^{l(k)-1}\|=\sum_{k=k_0}^v \|x^{l(k+m+1)}-x^{l(k+m+1)-1}\| \\
&\leq \frac{1}{4}\sum_{k=k_0}^v\|x^{l(k)}-x^{l(k)-1}\|+\tau\sum_{k=k_0}^v \chi\big(q(x^{l(k)}-q(\bar x)\big)-\chi\big(q(x^{l(k+m+1)})-q(\bar x)\big)\\
& \leq  \frac{1}{4}\sum_{k=k_0}^v\|x^{l(k)}-x^{l(k)-1}\|+\tau\sum_{k=k_0}^{k_0+m} \chi\big(q(x^{l(k)})-q(\bar x)\big),
\end{aligned}
\end{equation*}
equivalently,
\begin{equation*}
\frac{3}{4}\sum_{k=k_0+m+1}^{v+m+1} \|x^{l(k)}-x^{l(k)-1}\| \leq \frac{1}{4}\sum_{k=k_0}^{k_0+m}\|x^{l(k)}-x^{l(k)-1}\|+\tau\sum_{k=k_0}^{k_0+m} \chi\big(q(x^{l(k)})-q(\bar x)\big),
\end{equation*}
which implies from \eqref{Eq:maxNonStepCrit} and \cref{Pro:maxprevious} \ref{item:maxdecreasingPhi} that
\begin{equation*}
\begin{aligned}
\frac{3}{4}\sum_{k=k_0}^{v+m+1} \|x^{l(k)}-x^{l(k)-1}\|
&\leq \sum_{k=k_0}^{k_0+m}\|x^{l(k)}-x^{l(k)-1}\|+\tau\sum_{k=k_0}^{k_0+m} \chi\big(q(x^{l(k)})-q(\bar x)\big)\\
& \leq \sum_{k=k_0}^{k_0+m} \sqrt{\frac{2\big(q(x^{l(l(k)-1)})-q(x^{l(k)})\big)}{\delta \gamma_{\min}}}+\tau\sum_{k=k_0}^{k_0+m} \chi\big(q(x^{l(k)})-q(\bar x)\big)\\
& \leq m\sqrt{\frac{2\big(q(x^{l(l(k_0)-1)})-q(x^{l(k_0+m)})\big)}{\delta \gamma_{\min}}}+m\tau\chi\big(q(x^{l(k_0)})-q(\bar x)\big)\\
& \leq m\sqrt{\frac{2\big(q(x^{l(k_0-m-1)})-q(x^{l(k_0+m)})\big)}{\delta \gamma_{\min}}}+m\tau\chi\big(q(x^{l(k_0)})-q(\bar x)\big).
\end{aligned}
\end{equation*}
Hence, we have
\begin{equation*}
\sum_{k=k_0}^{v} \|x^{l(k)}-x^{l(k)-1}\| \leq \frac{4m}{3}\left( \sqrt{\frac{2\big(q(x^{l(k_0-m-1)})-q(x^{l(k_0+m)})\big)}{\delta \gamma_{\min}}}+\tau \chi\big(q(x^{l(k_0)})-q(\bar x)\big) \right).
\end{equation*}
\end{proof}
\begin{theorem}\label{Thmnon:maxGlobConv}
Let \cref{Ass:ProxGradNonMonotone} and \cref{Ass:Cont} 
 hold, and let $ \{ x^k \}_{k\in \mathbb N} $
be any sequence generated by \Cref{Alg:maxNonMonotoneProxGrad}.
Suppose that $ \{ x^k \}_{k\in \K} $ is a subsequence converging to
some limit point $ \bar x $, and that $ q $ has the KL property
at $\bar x$, then the entire sequence $ \{ x^k \}_{k\in \mathbb N} $ converges
to $ \bar x$.
\end{theorem}
\begin{proof}
By \cref{Pro:maxprevious}~\ref{item:maxdecreasingPhi}, one knows that the whole sequence $\{q(x^{l(k)})\}$ is monotonically decreasing and convergent to $q(\bar x)$ by assumption. It implies that $q(x^{l(k)}) \geq q(\bar x)$ for all $k\in \mathbb N$. If $q(x^{l(\bar k)})=q(\bar x)$ holds for some index $\bar k\in \mathbb N$, which, by monotonicity, implies that $q(x^{l(\bar k+1)})=q(\bar x)$, then we claim that $x^{k+1}=x^k$ for all $k \geq \bar k+m_{\bar k}$ ($m_{\bar k}=m$ if $\bar k$ is sufficiently large) by \cite[Lemma~3.2~(iv)]{doi:10.1137/22M1469663}.
Since the subsequence $\{x^k\}_{k\in \K}$ is assumed to converge to $\bar x$, this implies that $x^k=\bar x$ for all $k \in \mathbb N$ sufficiently large.  Then the entire (eventually constant) sequence $\{x^k\}_{k\in \mathbb N}$ is convergent to $\bar x$ in this situation.\\
\indent It remains to consider the case where $q(x^{l(k)})>q(\bar x)$ for all $k\in \mathbb N$.  Recall the analysis in \cref{Thmnon:maxauxConv}, let us assume again that $k_0 \geq \hat k$
  \ and $k_0\in K$ is sufficiently large to satisfy
\begin{equation}\label{Eq:Tmaxk0}
q(x^{l(k_0)})<q(\bar x)+\eta.
\end{equation}
We now claim that the following two statements hold for all $k\geq k_0$:
\begin{enumerate}[(a)]
	\item \label{Itemnon:maxInd-1}
	   $ x^k \in B_{\alpha} (\bar x) $, 
	\item \label{Itemnon:maxInd-2}
	  $ \| x^{k_0} - \bar x \| + \sum_{i=k_0}^k \| x^{i+1} - x^i \| \leq \alpha $, which is equivalent to
	  \begin{equation}\label{Eqnon:maxInd-2}
\begin{aligned}
	  	& \sum_{i=k_0}^k \| x^{i+1} - x^i \| \\
& \leq \left( \frac{8\sqrt{2}m\left(c_{\hat k}+1\right)(\bar \gamma_{\rho}+L_{\rho})}{3\delta \gamma_{\min}}\sqrt{\frac{\bar \gamma_{\rho}+L_{\rho}}{\mu}}+\frac{m}{3}\sqrt{\frac{2(\bar \gamma_{\rho}+L_{\rho})}{\delta \gamma_{\min}-\mu}}\left(\frac{3}{m}+\frac{2(\bar \gamma_{\rho}+L_{\rho})}{\delta \gamma_{\min}}\right)\right)\\
&~~~~\cdot \chi\big(q(x^{l(k_0)})-q(\bar x)\big)\\
&~~~~+ \left(\frac{4\sqrt{2}m\left(c_{\hat k}+1\right)}{3}\sqrt{\frac{\bar \gamma_{\rho}+L_{\rho}}{\mu}}+\frac{m}{3}\sqrt{\frac{2(\bar \gamma_{\rho}+L_{\rho})}{\delta \gamma_{\min}-\mu}}+\sqrt{\frac{\delta \gamma_{\min}}{\delta \gamma_{\min}-\mu}}\right)\\
&~~~~\cdot\sqrt{\frac{2\big(q(x^{l(k_0-m-1)})-q(\bar x)\big)}{\delta \gamma_{\min}}}\\
\end{aligned}
	  \end{equation}
\end{enumerate}
where $\alpha$ is defined in \eqref{EqNon:maxalpha}.
We still verify these two statements jointly by induction. For $k=k_0$, statement \ref{Itemnon:maxInd-1} holds by the definition of $\alpha$ in \eqref{EqNon:maxalpha}. Meanwhile,  due to $k_0\in \overline K$,  \eqref{Eq:definition for a} and \eqref{Eq:maxNonSubki} deduce that
\begin{equation*}
\begin{aligned}
\frac{\delta \gamma_{\min}-\mu}{2}\|x^{k_0+1}-x^{k_0}\|^2& \leq q(x^{l(k_0)})-q(x^{k_0+1})+q(x^{k_0+1})-q(x^{l(k_0+1)})\\
&=q(x^{l(k_0)})-q(x^{l(k_0+1)})\leq q(x^{l(k_0)})-q(\bar x)
\end{aligned}
\end{equation*}
which says that \eqref{Eqnon:maxInd-2} also holds for $k=k_0$.
Suppose that both statements are valid for all $k\geq k_0$. Using the triangle inequality, the induction hypothesis, and the definition of $\alpha$, we have
\begin{equation*}
\|x^{k+1}-\bar x \| \leq \sum_{i=k_0}^{k}\|x^{i+1}-x^i\|+\|x^{k_0}-\bar x\|\leq \alpha,
\end{equation*}
i.e., statement \ref{Itemnon:maxInd-1} holds for $k+1$ in place of $k$.  Therefore, we have $x^i \in B_{\alpha}(\bar x)$ for all $k_0 \leq i \leq k+1$ by induction. The verification of the induction step for \eqref{Eqnon:maxInd-2} is more involved.  We next consider two cases whether or not $k\in K$ defined in \eqref{Eq:definition for a} for all $k \geq k_0$.\\
\textbf{Case 1: $k \in \overline K$.}
For all $k$ satisfying $k \geq k_0$ and $k\in \overline K$, then one has 
\begin{equation}\label{Eq:case1.1}
0\leq q(x^{l(k+1)})-q(x^{k+1}) \leq \frac{\mu}{2}\|x^{k+1}-x^k\|^2.
\end{equation}
Since KL property holds at $\bar x$,  then \eqref{Eq:KLineq} with $m=0$ is also valid, i.e.,
\begin{equation}\label{Eq:KL1.1}
q(x^{l(i)})-q(x^{l(i+1)})
 \leq \left(\chi\big(q(x^{l(i)})-q(\bar x)\big)-\chi\big(q(x^{l(i+1)})-q(\bar x)\big)\right)(\bar \gamma_{\rho}+L_{\rho})\|x^{l(i)}-x^{l(i)-1}\|
\end{equation}
for all $ k_0 \leq i \leq k+1 $ and $i \in \overline K$.
\eqref{Eq:maxNonSubki} and \eqref{Eq:case1.1} imply that
\begin{equation}\label{Eq:fKL1.1}
\frac{\delta \gamma_{k}-\mu}{2}\|x^{i+1}-x^i\|^2 \leq q(x^{l(i)})-q(x^{i+1})+q(x^{i+1})-q(x^{l(i+1)})=q(x^{l(i)})-q(x^{l(i+1)}).
\end{equation}
Define $$\Delta_{i,j}:=\chi\big(q(x^{l(i)})-q(\bar x)\big)-\chi\big(q(x^{l(j)})-q(\bar x)\big)$$
for short. Then, \eqref{Eq:KL1.1} and \eqref{Eq:fKL1.1} yield that
\begin{equation*}
\begin{aligned}
\frac{\delta \gamma_{\min}-\mu}{2}\|x^{i+1}-x^i\|^2
\leq  (\bar \gamma_{\rho}+L_{\rho})\Delta_{i,i+1}\|x^{l(i)}-x^{l(i)-1}\|
\end{aligned}
\end{equation*}
for all $ k_0 \leq i \leq k+1 $ and $i \in \overline K$.  It implies from the inequality $\sqrt{ab} \leq \frac{1}{4} a+b$ for any $a \geq 0$ and $b \geq 0$ again that
\begin{equation*}
\begin{aligned}
\|x^{i+1}-x^i\| \leq \sqrt{\frac{2(\bar \gamma_{\rho}+L_{\rho})}{\delta \gamma_{\min}-\mu}}\sqrt{\Delta_{i,i+1}\|x^{l(i)}-x^{l(i)-1}\|}
 \leq \hat \tau \left(\Delta_{i,i+1}+\frac{1}{4}\|x^{l(i)}-x^{l(i)-1}\|\right)
\end{aligned}
\end{equation*}
for all $ k_0 \leq i \leq k+1 $ and $i \in \overline K$ with $\hat \tau:=\sqrt{\frac{2(\bar \gamma_{\rho}+L_{\rho})}{\delta \gamma_{\min}-\mu}}$.  Recall that $x^{i}\in B_{\alpha}(\bar x)$ for all $ k_0 \leq i \leq k+1$ by induction hypothesis and hence \cref{Thmnon:maxauxConv} holds with $v=k+1$.
Summation implies that
\begin{equation}\label{Eq:SumforOverlineK}
\begin{aligned}
&\sum^{k+1}_{k_0 = i \in \overline K} \leq \hat \tau \sum_{k_0 = i \in \overline K}^{k+1}\left(\Delta_{i,i+1}+\frac{1}{4}\|x^{l(i)}-x^{l(i)-1}\| \right)
 \leq \hat \tau \sum_{i = k_0}^{k+1}\left(\Delta_{i,i+1}+\frac{1}{4}\|x^{l(i)}-x^{l(i)-1}\| \right)\\
& \leq \hat \tau \chi\big(q(x^{l(k_0)})-q(\bar x)\big)+\hat \tau\frac{1}{4} \sum_{i = k_0}^{k+1}\|x^{l(i)}-x^{l(i)-1}\|\\
&\leq \hat \tau \chi\big(q(x^{l(k_0)})-q(\bar x)\big)\\
&~~~~+\frac{\hat \tau m}{3}\left( \sqrt{\frac{2\big(q(x^{l(k_0-m-1)})-q(\bar x)\big)}{\delta \gamma_{\min}}}+\frac{2(\bar \gamma_{\rho}+L_{\rho})}{\delta \gamma_{\min}}\chi\big(q(x^{l(k_0)})-q(\bar x)\big) \right) \\
& =\frac{m}{3}\sqrt{\frac{2(\bar \gamma_{\rho}+L_{\rho})}{\delta \gamma_{\min}-\mu}}\left(\sqrt{\frac{2\big(q(x^{l(k_0-m-1)})-q(\bar x)\big)}{\delta \gamma_{\min}}}+\left(\frac{3}{m}+\frac{2(\bar \gamma_{\rho}+L_{\rho})}{\delta \gamma_{\min}}\right)\chi\big(q(x^{l(k_0)})-q(\bar x)\big)\right).
\end{aligned}
\end{equation}
\textbf{Case 2: $k\in K$.}
For sufficiently large $k_0$, from \eqref{Eq:maxeta}, one has
\begin{equation}\label{Eq:distance}
\|x^{l(i+1)}-x^{i+1}\| \leq \sum_{j=i+1-m}^{i+1} \|x^{j+1}-x^j\| \leq m\eta, \quad \forall i \geq k_0,
\end{equation}
which means that $\|x^{l(i+1)}-x^{i+1}\| \to 0$.  By \eqref{Eq:maxNonSubki}, for each $i \in \mathbb N$, one has
\begin{equation*}
\begin{aligned}
&\left< \nabla f(x^{l(i+1)-1}), x^{l(i+1)}-x^{l(i+1)-1}\right> +\frac{\gamma_{l(i+1)-1}}{2}\|x^{l(i+1)}-x^{l(i+1)-1}\|^2+g(x^{l(i+1)})\\
&\leq \left< \nabla f(x^{l(i+1)-1}), x^{i+1}-x^{l(i+1)-1}\right> +\frac{\gamma_{l(i+1)-1}}{2}\|x^{i+1}-x^{l(i+1)-1}\|^2+g(x^{i+1}),
\end{aligned}
\end{equation*}
which implies that
\begin{equation}\label{Eq:sub3}
\begin{aligned}
q(x^{l(i+1)})-q(x^{i+1})&\leq \left< \nabla f(x^{l(i+1)-1}),x^{i+1}-x^{l(i+1)}\right>+f(x^{l(i+1)})-f(x^{i+1})\\
&~~~~+\frac{\gamma_{l(i+1)-1}}{2}\left(\|x^{i+1}-x^{l(i+1)-1}\|^2-\|x^{l(i+1)}-x^{l(i+1)-1}\|^2\right).
\end{aligned}
\end{equation}
Recall again that $x^i \in B_{\alpha}(\bar x)$ for all $k_0 \leq i \leq k+1$, then $x^{i+1} \in C_{\rho}$ holds and from 
\eqref{Eq:distance}, we have
\begin{equation*}
\begin{aligned}
\|x^{l(i+1)}-\bar x\| & \leq \| x^{l(i+1)}-x^{i+1}\|+\|x^{i+1}-x^i\|+ \|x^i-\bar x\|\\
& \leq (m+1)\eta+\alpha \leq \rho \quad  \forall k_0 \leq i \leq k+1,
\end{aligned}
\end{equation*}
and consequently
\begin{equation*}
\|x^{l(i+1)-1}-\bar x\| \leq \|x^{l(i+1)-1}-x^{l(i+1)}\|+\| x^{l(i+1)}-\bar x\|
\leq \eta+(m+1)\eta+\alpha<2\rho
\end{equation*}
for all $k_0 \leq i \leq k+1$.  Therefore,   for all $k_0 \leq i\leq k+1$, we have $\gamma_{l(i+1)-1} \leq \bar \gamma_{\rho}$ from \cref{Lem:maxboundedgamma},  and $x^i, x^{i+1}, x^{l(i+1)-1}, x^{l(i+1)} \in C_{\rho}$.  Then the descent lemma implies
\begin{equation}\label{Eq:dl}
f(x^{l(i+1)})-f(x^{i+1})\leq \frac{L_{\rho}}{2}\|x^{l(i+1)}-x^{i+1}\|^2+\left<\nabla f(x^{i+1}), x^{l(i+1)}-x^{i+1}\right>.
\end{equation}
Meanwhile, from the fact that $\frac{1}{2}\|\cdot-x^{l(i+1)-1}\|^2$ is convex, we have
\begin{equation}\label{Eq:convex}
\begin{aligned}
\frac{1}{2}\|x^{i+1}-x^{l(i+1)-1}\|^2-\frac{1}{2}\|x^{l(i+1)}-x^{l(i+1)-1}\|^2\leq \left<x^{i+1}-x^{l(i+1)-1},x^{i+1}-x^{l(i+1)}\right>
\end{aligned}
\end{equation}
Putting \eqref{Eq:sub3}, \eqref{Eq:dl}, and \eqref{Eq:convex} together yields that
\begin{equation}
\begin{aligned}
q(x^{l(i+1)})-q(x^{i+1})&\leq\left<\nabla f(x^{l(i+1)-1})-\nabla f(x^{i+1})+\gamma_{l(i+1)-1}\left(x^{i+1}-x^{l(i+1)-1}\right), x^{i+1}-x^{l(i+1)}\right>\\
&~~~~+\frac{L_{\rho}}{2}\|x^{l(i+1)}-x^{i+1}\|^2\\
& \leq\frac{L_{\rho}}{2}\|x^{l(i+1)}-x^{i+1}\|^2+\left(\bar \gamma_{\rho}+L_{\rho}\right)\|x^{i+1}-x^{l(i+1)-1}\|\|x^{i+1}-x^{l(i+1)}\|\\
& \leq \left(L_{\rho}+\frac{\bar \gamma_{\rho}}{2}\right)\|x^{l(i+1)}-x^{i+1}\|^2+\frac{\bar \gamma_{\rho}}{2}\|x^{l(i+1)}-x^{l(i+1)-1}\|^2,
\end{aligned}
\end{equation}
therefore, by the inequality $\sqrt{a+b}\leq \sqrt{a}+\sqrt{b}$ for all $a\geq 0$ and $b \geq 0$,  as well as \cref{Prop:Cons}, one has
\begin{equation}
\begin{aligned}
\sqrt{q(x^{l(i+1)})-q(x^{i+1})}&\leq \sqrt{L_{\rho}+\frac{\bar \gamma_{\rho}}{2}}\|x^{l(i+1)}-x^{i+1}\|+\sqrt{\frac{\bar \gamma_{\rho}}{2}}\|x^{l(i+1)}-x^{l(i+1)-1}\|\\
&\leq \left(c_{\hat k}\sqrt{L_{\rho}+\frac{\bar \gamma_{\rho}}{2}}+\sqrt{\frac{\bar \gamma_{\rho}}{2}}\right)\|x^{l(i+1)}-x^{l(i+1)-1}\|
\end{aligned}
\end{equation}
for all $k_0 \leq i \leq k+1$.
Then, we have
\begin{equation*}
\|x^{i+1}-x^i\| < \sqrt{\frac{2}{\mu}}\sqrt{q(x^{l(i+1)})-q(x^{i+1})} \leq \left( c_{\hat k}\sqrt{\frac{2L_{\rho}+\bar \gamma_{\rho}}{\mu}}+\sqrt{\frac{\bar \gamma_{\rho}}{\mu}}\right)\|x^{l(i+1)-1}-x^{l(i+1)}\|
\end{equation*}
for all $ k_0 \leq i \leq k+1$ and $i\in K$, which definitely says that
\begin{equation*}
\|x^{i+1}-x^i\| \leq \sqrt{2}\left(c_{\hat k}+1\right) \sqrt{\frac{\bar \gamma_{\rho}+L_{\rho}}{\mu}}\|x^{l(i+1)-1}-x^{l(i+1)}\| \quad \forall k_0 \leq i \leq k+1 \ \text{and} \ i \in K.
\end{equation*}
 We know that the induction hypothesis yields \cref{Thmnon:maxauxConv} holds automatically for $v=k+1$,  hence,  we have
\begin{equation}\label{Eq:sumK}
\begin{aligned}
&\sum_{k_0 = i \in K}^{k+1}\|x^{i+1}-x^i\| \leq   \sqrt{2}\left(c_{\hat k}+1\right) \sqrt{\frac{\bar \gamma_{\rho}+L_{\rho}}{\mu}}\sum_{k_0 = i \in K}^{k+1}\|x^{l(i+1)-1}-x^{l(i+1)}\|\\
 &\leq   \sqrt{2}\left(c_{\hat k}+1\right) \sqrt{\frac{\bar \gamma_{\rho}+L_{\rho}}{\mu}} \sum_{ i=k_0 }^{k+1}\|x^{l(i+1)-1}-x^{l(i+1)}\| \\
&\leq \frac{4\sqrt{2}m\left(c_{\hat k}+1\right)}{3} \sqrt{\frac{\bar \gamma_{\rho}+L_{\rho}}{\mu}}\left( \sqrt{\frac{2\big(q(x^{l(k_0-m-1)})-q(\bar x)\big)}{\delta \gamma_{\min}}}+\frac{2(\bar \gamma_{\rho}+L_{\rho})}{\delta \gamma_{\min}}\chi\big(q(x^{l(k_0)})-q(\bar x)\big) \right).
\end{aligned}
\end{equation}
By \eqref{Eq:SumforOverlineK} and \eqref{Eq:sumK}, we have
\begin{equation*}
\begin{aligned}
&\sum_{i=k_0 }^{k+1}\|x^{i+1}-x^i\| \leq \sum_{k_0 = i \in K}^{k+1}\|x^{i+1}-x^i\|+\sum_{k_0 = i \in \overline K}^{k+1}\|x^{i+1}-x^i\|\\
& \leq \left( \frac{8\sqrt{2}m\left(c_{\hat k}+1\right)(\bar \gamma_{\rho}+L_{\rho})}{3\delta \gamma_{\min}}\sqrt{\frac{\bar \gamma_{\rho}+L_{\rho}}{\mu}}+\frac{m}{3}\sqrt{\frac{2(\bar \gamma_{\rho}+L_{\rho})}{\delta \gamma_{\min}-\mu}}\left(\frac{3}{m}+\frac{2(\bar \gamma_{\rho}+L_{\rho})}{\delta \gamma_{\min}}\right)\right)\\
&~~~~\cdot \chi\big(q(x^{l(k_0)})-q(\bar x)\big)\\
&~~~~+ \left(\frac{4\sqrt{2}m\left(c_{\hat k}+1\right)}{3}\sqrt{\frac{\bar \gamma_{\rho}+L_{\rho}}{\mu}}+\frac{m}{3}\sqrt{\frac{2(\bar \gamma_{\rho}+L_{\rho})}{\delta \gamma_{\min}-\mu}}\right)\sqrt{\frac{2\big(q(x^{l(k_0-m-1)})-q(\bar x)\big)}{\delta \gamma_{\min}}}\\
& \leq \left( \frac{8\sqrt{2}m\left(c_{\hat k}+1\right)(\bar \gamma_{\rho}+L_{\rho})}{3\delta \gamma_{\min}}\sqrt{\frac{\bar \gamma_{\rho}+L_{\rho}}{\mu}}+\frac{m}{3}\sqrt{\frac{2(\bar \gamma_{\rho}+L_{\rho})}{\delta \gamma_{\min}-\mu}}\left(\frac{3}{m}+\frac{2(\bar \gamma_{\rho}+L_{\rho})}{\delta \gamma_{\min}}\right)\right)\\
&~~~~\cdot \chi\big(q(x^{l(k_0)})-q(\bar x)\big)\\
&~~~~+ \left(\frac{4\sqrt{2}m\left(c_{\hat k}+1\right)}{3}\sqrt{\frac{\bar \gamma_{\rho}+L_{\rho}}{\mu}}+\frac{m}{3}\sqrt{\frac{2(\bar \gamma_{\rho}+L_{\rho})}{\delta \gamma_{\min}-\mu}}+\sqrt{\frac{\delta \gamma_{\min}}{\delta \gamma_{\min}-\mu}}\right)\\
&~~~~\cdot\sqrt{\frac{2\big(q(x^{l(k_0-m-1)})-q(\bar x)\big)}{\delta \gamma_{\min}}}\\
\end{aligned}
\end{equation*}
Hence, statement \ref{Itemnon:maxInd-2} holds for $k+1$, and this completes the induction.

In particular, it follows from \ref{Itemnon:maxInd-1} that $x^k\in B_{\alpha}(\bar x)$ for all $k \geq k_0$. Taking $k \to \infty$ in \eqref{Eqnon:maxInd-2} therefore shows that $\{x^k\}$ is a Cauthy sequence and, thus, convergent. Since we already know that $\bar x$ is an accumulation point, it follows that the entire sequence $\{x^k\}_{k\in \mathbb N}$ is convergent to $\bar x$.
\end{proof}
Note that if $\overline K=\mathbb N$ or $m=0$, then we have $q(x^k)>q(x^{k+1})$ for all $k\in \mathbb N$, i.e., $\{q(x^k)\}_{k\in \mathbb N}$ is decreasing, then \Cref{Alg:maxNonMonotoneProxGrad} degenerates into \cite[Algorithm~3.1]{jia2023convergence}, where the specific results about the convergence and convergence rate of the whole sequence were proposed. 
In the following,  we illustrate the  convergence rate of \Cref{Alg:maxNonMonotoneProxGrad} where $\overline K\neq \mathbb N$.
\begin{theorem}\label{Thm:maxRate-of-Conv}
Let \cref{Ass:ProxGradNonMonotone} and \cref{Ass:Cont} hold, and let $ \{ x^k \}_{k\in \mathbb N} $
be any sequence generated by \Cref{Alg:maxNonMonotoneProxGrad}.
Suppose that $ \{ x^k \}_{k\in \K} $ is a subsequence converging to
some limit point $ \bar x $, and that $ q $ has the KL property
at $ \bar x $. Then the entire sequence $ \{ x^k \}_{k\in \mathbb N} $ converges
to $ \bar x $, and if the corresponding desingularization function
has the form $ \chi (t) = c t^{\theta} $ for some $ c > 0 $ and $\theta \in (0,1]$, then
the following statements hold:
\begin{enumerate}[(i)]
          \item \label{conv rate equal 1} if $\theta=1$, then the sequences $\{q(x^{l(k)})\}_{k\in \mathbb N}$ and $\{x^k\}_{k\in \mathbb N}$ converge in a finite numer of steps to $q(\bar x)$ and $\bar x$, respectively.
         \item \label{conv rate les 0.5} if $\theta\in [\frac{1}{2}, 1)$,  then the sequence $\{q(x^{l(k)})\}_{k\in \mathbb N}$ converges Q-linearly to $q(\bar x)$, and the sequence $\{x^k\}_{k\in \mathbb N}$ converges R-linearly to $\bar x$.
        \item \label{conv rate more 0.5} if $\theta\in (0, \frac{1}{2})$, the there exist some positive constants $\eta_1$ and $\eta_2$ such that
           \begin{align*}
           q(x^{l(k)})-q(\bar x) &\leq  \eta_1 k^{-\frac{1}{1-2\theta}}\\
          \|x^k-\bar x\| &\leq  \eta_2 k^{-\frac{\theta}{1-2\theta}}
             \end{align*}
           for sufficiently large $k$.
\end{enumerate}
\end{theorem}
\begin{proof}
Without loss of generality, we assume that $q(x^{l(k)})>q(\bar x)$ for all $k\in \mathbb N$. Therefore, for all $k \geq k_0$ which is defined by \cref{LemNon:maxalpha-small}, one has \eqref{Eq:maxk0} holds, and then KL property of $q$ at $\bar x$ and \cref{LemNon:maxDistanceSubgrad} imply that
\begin{equation}\label{Eq:KLlast}
\begin{aligned}
1 &\leq \chi'\big(q(x^{l(k)})-q(\bar x)\big) \dist\big(0,\partial q(x^{l(k)})\big)\\
& \leq c\theta \big(q(x^{l(k)})-q(\bar x)\big)^{\theta-1}(\bar \gamma_{\rho}+L_{\rho})\|x^{l(k)}-x^{l(k)-1}\| \quad \forall k\geq k_0.
\end{aligned}
\end{equation}
\eqref{Eq:preplusKL} and \eqref{Eq:KLlast} imply that
\begin{equation*}
\begin{aligned}
q(x^{l(k+1)})-q(x^{l(k)})&\leq -\delta \frac{\gamma_{\min}}{2}\|x^{l(k+1)}-x^{l(k+1)-1}\|^2\\
& \leq -\delta \frac{\gamma_{\min}}{2(c\theta)^2(\bar \gamma_{\rho}+L_{\rho})^2}\big(q(x^{l(k+1)})-q(\bar x)\big)^{2(1-\theta)}\\
& :=-\xi\big(q(x^{l(k+1)})-q(\bar x)\big)^{2(1-\theta)} \quad \forall k\geq k_0
\end{aligned}
\end{equation*}
with $\xi:=\delta \frac{\gamma_{\min}}{2(c\theta)^2(\bar \gamma_{\rho}+L_{\rho})^2}$. Hence, we have
\begin{equation}\label{Eq:Relq}
\big(q(x^{l(k+1)})-q(\bar x)\big)^{2(1-\theta)} \leq \frac{1}{\xi}\big(q(x^{l(k)})-q(x^{l(k+1)})\big) \quad \forall k\geq k_0.
\end{equation}
Since $\{q(x^{l(k)})\}_{k\in \mathbb N}$ is decreasing, then statements \ref{conv rate equal 1},  \ref{conv rate les 0.5}, and \ref{conv rate more 0.5} regarding the sequence $\{q(x^{l(k)})\}_{k\in \mathbb N}$ follow from \cite[Lemma~2.14]{jia2023augmented}.  It remains to verify the convergence rate with respect to the sequence $\{x^k\}_{k\in \mathbb N}$.  We now consider the different cases of $\theta$.
\begin{itemize}
\item $\theta=1$: for all $k \geq k_0$, \eqref{Eq:Relq} implies that $q(x^{l(k+1)})-q(x^{l(k)})\leq -\xi$, which yields that
$\{q(x^{l(k)})\}_{k\in \mathbb N}$ converges to $q(\bar x)$ in finite steps, then one has $\{x^k\}_{k\in \mathbb N}$ also converges to $\bar x$ in finite steps.
\item $\theta\in (0, 1/2)$: for all $k \geq k_0$ and $s>k$, \eqref{Eqnon:maxInd-2} implies that
\begin{equation*}
\begin{aligned}
\|x^k-x^s\| &\leq \sum_{i=k}^{s}\|x^{i+1}-x^i\| \leq a \sqrt{q(x^{l(k-m-1)})-q(\bar x)}+bc \big(q(x^{l(k)})-q(\bar x) \big)^{\theta}\\
&\leq (a+bc)\big(q(x^{l(k-m-1)})-q(\bar x)\big)^{\theta},
\end{aligned}
\end{equation*}
where $a=\left(\frac{4m}{3}\sqrt{\frac{\bar \gamma_{\rho}+L_{\rho}}{\mu}}+\frac{m}{3}\sqrt{\frac{2(\bar \gamma_{\rho}+L_{\rho})}{\delta \gamma_{\min}-\mu}}+\sqrt{\frac{\delta\gamma_{\min} }{\delta\gamma_{\min}-\mu}}\right)\sqrt{\frac{2}{\delta \gamma_{\min}}}$, $b=\frac{8m(\bar \gamma_{\rho}+L_{\rho})}{3\delta \gamma_{\min}}\sqrt{\frac{\bar \gamma_{\rho}+L_{\rho}}{\mu}}+\frac{m}{3}\sqrt{\frac{2(\bar \gamma_{\rho}+L_{\rho})}{\delta \gamma_{\min}-\mu}}\left(\frac{3}{m}+\frac{2(\bar \gamma_{\rho}+L_{\rho})}{\delta \gamma_{\min}}\right)$.\\
Note that there exists some constant $\eta_1$ such that
\begin{equation*}
q(x^{l(k)})-q(\bar x) \leq \eta_1 k^{\frac{1}{2\theta-1}} \quad \forall \ \text{sufficiently large}\ k,
\end{equation*}
then
\begin{equation*}
\|x^k-x^s\| \leq (a+b)\eta_1^{\theta}k^{\frac{\theta}{2\theta-1}}:=\eta_2k^{\frac{\theta}{2\theta-1}} \quad \forall \ \text{sufficiently large}\ k,
\end{equation*}
where $\eta_2:=(a+b)\eta_1^{\theta}$, taking $s\to \infty$ yields the desired result.
\item $\theta\in [1/2,1)$: for all $k \geq k_0$ and $s>k$,  \eqref{Eqnon:maxInd-2} implies that
\begin{equation*}
\begin{aligned}
\|x^k-x^s\| &\leq \sum_{i=k}^{s}\|x^{i+1}-x^i\| \leq a \sqrt{q(x^{l(k-m-1)})-q(\bar x)}+bc \big(q(x^{l(k)})-q(\bar x) \big)^{\theta}\\
&\leq (a+bc)\big(q(x^{l(k-m-1)})-q(\bar x)\big)^{\frac{1}{2}}.
\end{aligned}
\end{equation*}
Note that $\{q(x^{l(k)})\}_{k\in \mathbb N}$ Q-linearly converges to $q(\bar x)$,  then $\{x^k\}_{k\in \mathbb N}$ R-linearly converges to $\bar x$ (by taking $s\to \infty$).
\end{itemize}
This completes the proof.
\end{proof}

\section{Conclusion}\label{Sec:con}
This manuscript presented the convergence and rate-of-convergence results of two types of NPG methods, i.e., those combined with average line search and max line search, respectively. The results were established under the KL property,  requiring only that the smooth part of the objective function has locally (rather than globally) Lipschitz continuous gradients. Note that the partitioning of the index set is a useful tool for the convergence analysis of NPG methods.  At least for the NPG method with average line search,  the convergence theory  is independent on the specific choice of partitioning.  Additionally,  our results in \cref{Sec:max} addressed the question raised in \cite[Final Remarks]{kanzow2024convergence}, demonstrating that the technique of proof there can be extended to the NPG method with max line search.

\bibliographystyle{habbrv}
\bibliography{references}

\begin{thebibliography}{10}
\expandafter\ifx\csname url\endcsname\relax
  \def\url#1{\texttt{#1}}\fi
\expandafter\ifx\csname doi\endcsname\relax
  \def\doi#1{\burlalt{doi:#1}{http://dx.doi.org/#1}}\fi
\expandafter\ifx\csname urlprefix\endcsname\relax\def\urlprefix{URL }\fi
\expandafter\ifx\csname href\endcsname\relax
  \def\href#1#2{#2}\fi
\expandafter\ifx\csname burlalt\endcsname\relax
  \def\burlalt#1#2{\href{#2}{#1}}\fi

\bibitem{attouch2009convergence}
H.~Attouch and J.~Bolte.
\newblock On the convergence of the proximal algorithm for nonsmooth functions
  involving analytic features.
\newblock {\em Mathematical Programming}, 116(1):5--16, 2009.
\newblock \doi{10.1007/s10107-007-0133-5}.

\bibitem{AttouchBolteRedontSoubeyran2010}
H.~Attouch, J.~Bolte, P.~Redont, and A.~Soubeyran.
\newblock Proximal alternating minimization and projection methods for
  nonconvex problems: An approach based on the {K}urdyka-{{\L}}ojasiewicz
  inequality.
\newblock {\em Mathematics of Operations Research}, 35(2):438--457, 2010.
\newblock \doi{10.1287/moor.1100.0449}.

\bibitem{AttouchBolteSvaiter2013}
H.~Attouch, J.~Bolte, and B.~F. Svaiter.
\newblock Convergence of descent methods for semi-algebraic and tame problems,
  proximal algorithms, forward-backward splitting, and regularized
  {G}auss--{S}eidel methods.
\newblock {\em Mathematical Programming}, 137:91 -- 129, 2013.
\newblock \doi{10.1007/s10107-011-0484-9}.

\bibitem{BianChen2015}
W.~Bian and X.~Chen.
\newblock Linearly constrained non-{L}ipschitz optimization for image
  restoration.
\newblock {\em SIAM Journal on Imaging Sciences}, 8(4):2294--2322, 2015.
\newblock \doi{10.1137/140985639}.

\bibitem{BolteDaniilidisLewis2007}
J.~Bolte, A.~Daniilidis, and A.~Lewis.
\newblock The {{\L}}ojasiewicz inequality for nonsmooth subanalytic functions
  with applications to subgradient dynamical systems.
\newblock {\em SIAM Journal on Optimization}, 17(4):1205--1223, 2007.
\newblock \doi{10.1137/050644641}.

\bibitem{BolteDaniilidisLewisShiota2007}
J.~Bolte, A.~Daniilidis, A.~Lewis, and M.~Shiota.
\newblock Clarke subgradients of stratifiable functions.
\newblock {\em SIAM Journal on Optimization}, 18(2):556--572, 2007.
\newblock \doi{10.1137/060670080}.

\bibitem{BolteSabachTeboulle2014}
J.~Bolte, S.~Sabach, and M.~Teboulle.
\newblock Proximal alternating linearized minimization for nonconvex and
  nonsmooth problems.
\newblock {\em Mathematical Programming}, 146:459 -- 494, 2014.
\newblock \doi{10.1007/s10107-013-0701-9}.

\bibitem{BolteSabachTeboulleVaisbourd2018}
J.~Bolte, S.~Sabach, M.~Teboulle, and Y.~Vaisbourd.
\newblock First order methods beyond convexity and {L}ipschitz gradient
  continuity with applications to quadratic inverse problems.
\newblock {\em SIAM Journal on Optimization}, 28(3):2131--2151, 2018.
\newblock \doi{10.1137/17M1138558}.

\bibitem{bonettini2018block}
S.~Bonettini, M.~Prato, and S.~Rebegoldi.
\newblock A block coordinate variable metric linesearch based proximal gradient
  method.
\newblock {\em Computational Optimization and Applications}, 71(1):5--52, 2018.
\newblock \doi{10.1007/s10589-018-0011-5}.

\bibitem{BotCsetnek2016}
R.~I. Bo{\c{t}} and E.~R. Csetnek.
\newblock An inertial {T}seng’s type proximal algorithm for nonsmooth and
  nonconvex optimization problems.
\newblock {\em Journal of Optimization Theory and Applications},
  171(2):600--616, 2016.
\newblock \doi{10.1007/s10957-015-0730-z}.

\bibitem{BotCsetnekLaszlo2016}
R.~I. Bo{\c{t}}, E.~R. Csetnek, and S.~C. L{\'a}szl{\'o}.
\newblock An inertial forward--backward algorithm for the minimization of the
  sum of two nonconvex functions.
\newblock {\em EURO Journal on Computational Optimization}, 4(1):3--25, 2016.
\newblock \doi{10.1007/s13675-015-0045-8}.

\bibitem{BrucksteinDonohoElad2009}
A.~M. Bruckstein, D.~L. Donoho, and M.~Elad.
\newblock From sparse solutions of systems of equations to sparse modelling of
  signals and images.
\newblock {\em SIAM Review}, 51(1):34--81, 2009.
\newblock \doi{10.1137/060657704}.

\bibitem{de2023proximal}
A.~De~Marchi.
\newblock Proximal gradient methods beyond monotony.
\newblock {\em Journal of Nonsmooth Analysis and Optimization}, 4(10290), 2023.
\newblock \doi{10.46298/jnsao-2023-10290}.

\bibitem{grippo1986nonmonotone}
L.~Grippo, F.~Lampariello, and S.~Lucidi.
\newblock A nonmonotone line search technique for {N}ewton’s method.
\newblock {\em SIAM Journal on Numerical Analysis}, 23(4):707--716, 1986.
\newblock \doi{10.1137/0723046}.

\bibitem{jia2023augmented}
X.~Jia.
\newblock {\em Augmented Lagrangian Methods invoking (Proximal) Gradient-type
  Methods for (Composite) Structured Optimization Problems}.
\newblock PhD thesis, Universit{\"a}t W{\"u}rzburg, 2023.
\newblock
  \urlprefix\url{https://opus.bibliothek.uni-wuerzburg.de/frontdoor/index/index/docId/32374}.

\bibitem{jia2023convergence}
X.~Jia, C.~Kanzow, and P.~Mehlitz.
\newblock Convergence analysis of the proximal gradient method in the presence
  of the {K}urdyka–Łojasiewicz property without global lipschitz
  assumptions.
\newblock {\em SIAM Journal on Optimization}, 33(4):3038--3056, 2023.
\newblock \doi{10.1137/23M1548293}.

\bibitem{JiaKanzowMehlitzWachsmuth2021}
X.~Jia, C.~Kanzow, P.~Mehlitz, and G.~Wachsmuth.
\newblock An augmented {L}agrangian method for optimization problems with
  structured geometric constraints.
\newblock {\em Mathematical Programming}, 2022.
\newblock \doi{10.1007/s10107-022-01870-z}.

\bibitem{kanzow2024convergence}
C.~Kanzow and L.~Lehmann.
\newblock Convergence of nonmonotone proximal gradient methods under the
  {K}urdyka–Łojasiewicz property without a global lipschitz assumption.
\newblock {\em arXiv preprint arXiv:2411.12376}, 2024.
\newblock \doi{10.48550/arXiv.2411.12376}.

\bibitem{KanzowMehlitz2022}
C.~Kanzow and P.~Mehlitz.
\newblock Convergence properties of monotone and nonmonotone proximal gradient
  methods revisited.
\newblock {\em Journal of Optimization Theory and Applications},
  195(2):624--646, 2022.
\newblock \doi{10.1007/s10957-022-02101-3}.

\bibitem{Kurdyka1998}
K.~Kurdyka.
\newblock On gradients of functions definable in o-minimal structures.
\newblock {\em Annales de l'institut Fourier}, 48(3):769--783, 1998.
\newblock \doi{10.5802/aif.1638}.

\bibitem{NIPS2015_f7664060}
H.~Li and Z.~Lin.
\newblock Accelerated proximal gradient methods for nonconvex programming.
\newblock In {\em Advances in Neural Information Processing Systems},
  volume~28. Curran Associates, Inc., 2015.
\newblock
  \urlprefix\url{https://proceedings.neurips.cc/paper_files/paper/2015/file/f7664060cc52bc6f3d620bcedc94a4b6-Paper.pdf}.

\bibitem{liu2024nonmonotone}
H.~Liu, T.~Wang, and Z.~Liu.
\newblock A nonmonotone accelerated proximal gradient method with variable
  stepsize strategy for nonsmooth and nonconvex minimization problems.
\newblock {\em Journal of Global Optimization}, pages 1--35, 2024.
\newblock \doi{10.1007/s10898-024-01366-4}.

\bibitem{Mordukhovich2018}
B.~S. Mordukhovich.
\newblock {\em Variational Analysis and Applications}.
\newblock Springer, 2018.
\newblock \doi{10.1007/978-3-319-92775-6}.

\bibitem{Ochs2018}
P.~Ochs.
\newblock Local convergence of the heavy-ball method and i{P}iano for
  non-convex optimization.
\newblock {\em Journal of Optimization Theory and Applications},
  177(1):153--180, 2018.
\newblock \doi{10.1007/s10957-018-1272-y}.

\bibitem{OchsChenBroxPock2014}
P.~Ochs, Y.~Chen, T.~Brox, and T.~Pock.
\newblock i{P}iano: Inertial proximal algorithm for nonconvex optimization.
\newblock {\em SIAM Journal on Imaging Sciences}, 7(2):1388--1419, 2014.
\newblock \doi{10.1137/130942954}.

\bibitem{doi:10.1137/22M1469663}
Y.~Qian and S.~Pan.
\newblock Convergence of a class of nonmonotone descent methods for
  {K}urdyka–Łojasiewicz optimization problems.
\newblock {\em SIAM Journal on Optimization}, 33(2):638--651, 2023.
\newblock \doi{10.1137/22M1469663}.

\bibitem{qian2024convergence}
Y.~Qian, T.~Tao, S.~Pan, and H.~Qi.
\newblock Convergence of {Z}{H}-type nonmonotone descent method for
  {K}urdyka–Łojasiewicz optimization problems.
\newblock {\em arXiv preprint arXiv:2406.05740}, 2024.
\newblock \doi{10.48550/arXiv.2406.05740}.

\bibitem{wang2024class}
T.~Wang and H.~Liu.
\newblock A class of modified accelerated proximal gradient methods for
  nonsmooth and nonconvex minimization problems.
\newblock {\em Numerical Algorithms}, 95(1):207--241, 2024.
\newblock \doi{10.1007/s11075-023-01569-y}.

\bibitem{doi:10.1137/S1052623403428208}
H.~Zhang and W.~W. Hager.
\newblock A nonmonotone line search technique and its application to
  unconstrained optimization.
\newblock {\em SIAM Journal on Optimization}, 14(4):1043--1056, 2004.
\newblock \doi{10.1137/S1052623403428208}.

\end{thebibliography}

\end{document}